%% file: HHLSSZ2.tex
\documentclass[12pt]{amsart}

\usepackage{amssymb}
\usepackage{theoremref}
\usepackage{subfigure}
\usepackage[mathlines]{lineno}
\usepackage[margin=1in, footskip=0.5in]{geometry}

\usepackage{scrextend}
\usepackage{tikz}
\usepackage[T1]{fontenc}
\usepackage{xcolor}

\usepackage{algorithmic}
\usepackage{algorithm}
\usepackage{mathtools}
\usepackage{comment}
\usepackage{footnote}
\usepackage{verbatim}
\numberwithin{figure}{section} 
\usepackage{url}
\usetikzlibrary{external}
\IfFileExists{Writeup/Figures/figure_nminus3.pdf}{\tikzsetexternalprefix{Writeup/Figures/}}{\tikzsetexternalprefix{Figures/}}
\usepackage{shellesc}
\ifnum\ShellEscapeStatus=1
\else
\fi

\newcommand{\dcup}{\mathbin{\,\sqcup\,}}

\DeclareMathOperator{\Z}{Z}
\DeclareMathOperator{\pt}{pt}
\DeclareMathOperator{\comp}{comp}

\DeclareMathOperator{\thr}{th}

\newcommand{\thp}{\thr_+}
\newcommand{\ptp}{\pt_+}
\newcommand{\Zp}{\Z_+}
\newcommand{\F}{\mathcal F}

\newtheorem{theorem}{Theorem}[section]
\newtheorem{lemma}[theorem]{Lemma}
\newtheorem{corollary}[theorem]{Corollary}
\newtheorem{proposition}[theorem]{Proposition}
\newtheorem{observation}[theorem]{Observation}

\newtheorem*{claim*}{Claim}
\theoremstyle{definition}

\newtheorem{remark}[theorem]{Remark}

\makeatletter
\usetikzlibrary{calc}
\newcount\gem@row\newcount\gem@col\newcount\gem@ct
\newdimen\Lpt\newdimen\Spt
\colorlet{grayishcyan}{green!30!blue!60!white!95!black}
\colorlet{midblue}{blue!50!white}
\colorlet{darkishblue}{blue!85!black}
\colorlet{gemNodeFill}{grayishcyan!50!midblue}
\colorlet{gemNodeDraw}{gemNodeFill!70!darkishblue}
\tikzset{set graph scales/.code n args={3}{\Lpt=#1pt\Spt=#2pt\tikzset{scale=#3}}}
\pgfkeys{/graphs en masse/.cd,gem node fmt/.style={inner sep=0pt,minimum size=1.8\Spt,line width=0.3\Lpt,circle,draw=gemNodeDraw,fill=gemNodeFill},gem edge fmt/.style={line width=0.25\Lpt,draw=black!75},grid fmt/.style={line width=0.2\Lpt,draw=black!35},vert sep/.initial=0.125,horz sep/.initial=0.125,add graph internal/.style 2 args={/graphs en masse/next index/#1/.get=\@graph@num,/graphs en masse/next index/#1/.evaluated={int(\@graph@num+1)},/graphs en masse/graphs/#1/\@graph@num/.code={\begin{scope}[##1,every node/.append style={/graphs en masse/gem node fmt},every edge/.append style={/graphs en masse/gem edge fmt}]#2\end{scope}}}}

\def\@parse@colrow#1x#2\end@parse{\def\@cols{#1}\def\@rows{#2}}
\newcommand\DrawGraphs[2]{\@parse@colrow#2\end@parse\pgfkeys{/graphs en masse/next index/#1/.get=\gem@n@tgts,/graphs en masse/vert sep/.get=\gem@vsep,/graphs en masse/horz sep/.get=\gem@hsep}\pgfmathsetmacro{\gem@cwid}{\gem@hsep*2+1}\pgfmathsetmacro{\gem@rht}{\gem@vsep*2+1}\pgfmathsetmacro{\gem@b@cols}{\gem@n@tgts-(\@rows-1)*\@cols-1}\draw[/graphs en masse/grid fmt,xstep=\gem@cwid,ystep=\gem@rht,line cap=rect] (0,\gem@rht) grid ($(\@cols*\gem@cwid,\@rows*\gem@rht)$) (0,0) grid ($(\gem@b@cols*\gem@cwid,\gem@rht)-(0,\pgflinewidth/2)$);\gem@row\@rows\advance\gem@row\m@ne\gem@col\z@\gem@ct\@ne\loop\ifnum\gem@ct<\gem@n@tgts\pgfkeys{/graphs en masse/graphs/#1/\the\gem@ct={shift={($(\the\gem@col*\gem@cwid+\gem@hsep,\the\gem@row*\gem@rht+\gem@vsep)$)}}}\advance\gem@col\@ne\ifnum\gem@col<\@cols\else\gem@col\z@\advance\gem@row\m@ne\fi\advance\gem@ct\@ne\repeat}
\makeatother

\pgfkeys{/graphs en masse/next index/n minus 3/.initial=1,/graphs en masse/add graph/.style={/graphs en masse/add graph internal={n minus 3}{#n minus 3}}}\input nminus3.tex\relax
\pgfkeys{/graphs en masse/next index/n minus 4/.initial=1,/graphs en masse/add graph/.style={/graphs en masse/add graph internal={n minus 4}{#n minus 4}}}\input nminus4.tex\relax

\newcommand{\ol}{\overline}

\newcommand{\bit}{\begin{itemize}}
\newcommand{\eit}{\end{itemize}}
\newcommand{\ben}{\begin{enumerate}}
\newcommand{\een}{\end{enumerate}}
\newcommand{\beq}{\begin{equation}}
\newcommand{\eeq}{\end{equation}}
\newcommand{\bea}{\begin{eqnarray*}}
\newcommand{\eea}{\end{eqnarray*}}
\newcommand{\bpf}{\begin{proof}}
\newcommand{\epf}{\end{proof}\ms}
\newcommand{\bmt}{\begin{bmatrix}}
\newcommand{\emt}{\end{bmatrix}}
\newcommand{\ms}{\medskip}

\newcommand{\lc}{\left\lceil}
\newcommand{\rc}{\right\rceil}

\newcommand{\noi}{\noindent}

\title{Upper bounds for positive semidefinite propagation time}

\author[L. Hogben]{Leslie Hogben}
\address{Department of Mathematics, Iowa State University, Ames, IA 50011 and American Institute of Mathematics, San Jose, CA 95112}
\email{hogben@aimath.org}
\author[M. Hunnell]{Mark Hunnell}
\address{Department of Mathematics, Winston-Salem State University, Winston-Salem, NC 27110}
\email{hunnellm@wssu.edu}
\author[K. Liu]{Kevin Liu}
\address{Department of Mathematics, University of Washington, Seattle, WA 98195}
\email{kliu15@uw.edu}
\author[H. Schuerger]{Houston Schuerger}
\address{Department of Mathematics, Trinity College, Hartford, CT 06106}
\email{houston.schuerger@trincoll.edu}
\author[B. Small]{Ben Small}
\address{University Place, WA 98466}
\email{bentsm@gmail.com}
\author[Y. Zhang]{Yaqi Zhang} 
\address{Department of Mathematics, Drexel University, Philadelphia, PA 19104}
\email{yz844@drexel.edu}

\begin{document}

\maketitle
\vspace{-10pt}\begin{center}\today\end{center}

\begin{abstract} The tight upper bound $\ptp(G)\leq \lc \frac{|V(G)|-\Zp(G)}{2} \rc$ 
is established for the positive semidefinite propagation time of a graph  in terms of its positive semidefinite zero forcing number.  To prove this bound, two methods of transforming one positive semidefinite zero forcing set into another and algorithms implementing these methods are presented.  Consequences of the bound, including a tight Nordhaus-Gaddum sum upper bound on  positive semidefinite propagation time, are established.
\end{abstract}\bigskip

\noi {\bf Keywords} PSD propagation time, PSD zero forcing, migration, Nordhaus-Gaddum 

\noi{\bf AMS subject classification} 05C57, 05C35, 05C69, 05C85

\section{Introduction}

Zero forcing was introduced in \cite{AIM} to provide an upper bound for the maximum nullity
of symmetric matrices described by a graph,  and independently in \cite{graphinfect} in the study of control of quantum systems.  Zero forcing starts with a set of blue vertices and uses a color change rule to color the remaining vertices blue (this is called forcing).  The propagation time of a graph was introduced formally in 2012 by Hogben et al.~\cite{proptime} and  Chilakamarri et al.~\cite{Chil12}.  The propagation time of a zero forcing set is the number of time steps needed to fully color a graph blue when performing independent forces simultaneously, and the propagation time of a graph is the minimum of the propagation times over minimum zero forcing sets.

Positive semidefinite (PSD) forcing was defined in \cite{smallparam} to provide an upper bound for the maximum nullity of positive semidefinite matrices described by a graph (precise definitions of PSD forcing and other terms used throughout  are given at the end of this introduction). PSD forcing was studied more extensively in \cite{ekstrand} and Warnberg introduced the study of PSD propagation time in \cite{warnberg}.  It is well known that the propagation time of a path is one less than its order, and other families of graphs attain propagation time close to the order of the graph.  However, the behavior for PSD propagation time is very different.  Warnberg showed in \cite{warnberg} that the number of graphs  that have propagation time at least $|V(G)|-2$ is finite, but did not provide an upper bound on PSD propagation time that is tight for graphs of arbitrarily large order.

In this paper, we give two proofs of a tight upper bound on the  PSD propagation time of a graph, 
\beq\label{eq-ub}\ptp(G)\leq \lc \frac{|V(G)|-\Zp(G)}{2} \rc.\eeq  
This bound generalizes the next (well-known) result.

\begin{remark}
\thlabel{TreePropTimeLemma}
If $T$ is a tree of order $n$, then  $\ptp(T) \leq \lc \frac{n-1}{2} \rc$
with equality when $T$ is a path, since a blue vertex PSD forces every white neighbor in a tree.
\end{remark}

The bound \eqref{eq-ub} implies that $\ptp(G)\leq \frac n 2$ for a graph of order $n$ and   that there are only a finite number of graphs having $\ptp(G)\ge |V(G)|-k$  for any fixed natural number $k$ (see Section \ref{s:consequences}). The techniques used to prove \eqref{eq-ub} involve transforming one PSD forcing set into another, thereby reducing the propagation time if it was greater than $\lc \frac{|V(G)|-\Zp(G)}{2} \rc$.  In Section \ref{svm} a single vertex in the PSD  forcing set is exchanged, whereas in Section \ref{mvm} multiple vertices are exchanged. Both these techniques are called migration.  Algorithms using migration methods to transform any minimum PSD forcing set into one that achieves the bound in \eqref{eq-ub} are presented in Sections \ref{svm} and \ref{mvm}.
In Section \ref{s:consequences} we also derive additional consequences of the bound \eqref{eq-ub}, including tight Nordhaus-Gaddum sum bounds on PSD propagation time.

In the remainder of this introduction we provide precise definitions and introduce notation. 
A (simple) \emph{graph} is denoted by $G=(V(G),E(G))$ where  $V(G)$ is the finite nonempty set of \emph{vertices} of $G$ and  $E(G)$ is the \emph{edge set} of $G$; an edge is a two-element set of vertices and the edge $\{u,v\}$ can also be denoted by $uv$ (or $vu$). 
The \emph{order} of $G$ is the number of vertices $|V(G)|$. In a graph, two vertices $u$ and $v$ are \emph{adjacent} if $uv$ is an edge and each of $u$ and $v$ is \emph{incident} to $uv$. The \emph{degree} $\deg_G(v)$ of a vertex $v$ is the number of vertices that are adjacent to $v$ in $G$; when the context is clear, the subscript is omitted.  Vertex $u$ is a \emph{neighbor} of $v$ if $vu\in E(G)$ and the \emph{neighborhood} of  $v$ is $N(v)=\{u\in V(G): vu\in E(G)\}$.  
 For $U\subseteq V(G)$, the \emph{subgraph  of $G$ induced by $U$}, denoted by $G[U]$, is the graph obtained from $G$ by removing all vertices not in $U$ and their incident edges; $G[U]$ is also called an 
\emph{induced subgraph} of $G$.  For $S\subseteq V(G)$, $G-S=G[V(G)\setminus S]$, i.e., the subgraph obtained from $G$ by deleting vertices in $S$ and incident edges.   A \emph{path in $G$} is a sequence of distinct vertices $v_0,v_1,\dots, v_k$ such that $v_{i-1}v_i\in E(G)$ for $i=1,\dots, k$.  A graph is \emph{connected} if for two distinct vertices $v$ and $u$ there is a path from $v$  to $u$ (and thus also from $u$ to $v$).  A \emph{component} of a graph is a maximal connected induced subgraph. A \emph{path} (or \emph{path graph}) $P_n$ is  a graph with $V(P_n)=\{v_1,\dots,v_n\}$ and $E(P_n)=\{v_iv_{i+1}: i=1,\ldots,n-1\}$.  
A \emph{complete graph} $K_n$   is a graph with $V(K_n)=\{v_1,\dots,v_n\}$ and $E(K_n)=\{v_iv_{j}: 1\le i<j\le n\}$.

Each variant of zero forcing   is a process.  At every stage of the process, each vertex is either blue or white.  A white vertex may change color to blue at some step, but a blue vertex will remain blue during all subsequent steps.  Each variant of zero forcing  is determined by a color change rule that defines when a vertex may change the color of a white vertex to blue, i.e., perform a force. 
  The \emph{standard  color change rule} is: A blue vertex $u$ can change the color of a white vertex $w$ to blue  if $w$ is the unique white neighbor of $u$.   A force $u\to w$ using the standard color change rule  is called a \emph{standard force}.
Let $B$ be the set of blue vertices (at a particular stage of the process),  and let $W_1,\dots, W_k$ be the sets of vertices of the $k\ge 1$ components of $G-B$.  
 The \emph{PSD color change rule} is: If $u\in  B$, $w\in W_i$, and $w$ is the only white neighbor of  $u$ in $G[W_i\cup B]$, then change the color of  $w$ to blue. A force $u\to w$ using the PSD color change rule is called a \emph{PSD force}. Note that it is possible that there is only one component of $G-B$, and in that case a PSD force is the same as a standard force.   A set that can color every vertex in the graph blue by repeated applications of the PSD forcing rule is a \emph{PSD forcing set} and the minimum cardinality of a PSD forcing set of $G$ is the \emph{PSD zero forcing number}  $\Zp(G)$. Given a  PSD forcing set $B$, a set  of forces $\F$ that colors the entire graph blue, and a vertex  $b\in B$, define $V_b$ to be the set of vertices  $w$ such that  there is a sequence of forces  $b=v_1\to v_2\to\dots\to v_k=w$ in $\F$ (the empty sequence of forces to reach $w$ is permitted, i.e., $b\in V_b$).
The \emph{forcing tree}  of $b$ is the induced subgraph $G[V_b]$.
 
Starting with a set $B\subseteq V(G)$ of blue vertices, we define two sequences of sets, the set $B^{(i)}$ of vertices that are forced (change color from white to blue) at time step $i$   and  the set $B^{[i]}$ of vertices that are blue after time step $i$. Thus $B^{[0]}=B^{(0)}=B$ is the set of  vertices that are blue initially and after each subsequent time step $i+1$ we have $B^{[i+1]}=B^{[i]}\cup B^{(i+1)}$. To construct $B^{(i+1)}$ (and thus $B^{[i+1]}$), if $B^{(i)}$ and $B^{[i]}$ have been determined, then 
\[B^{(i+1)}=\{w: 
\mbox{ $w$ can be  PSD forced by some vertex (given the vertices in  $B^{[i]}$ are blue)}\}.\] 
 The   \emph{PSD propagation time} of $B\subseteq V(G)$, denoted by $\ptp(G;B)$, is the least $t$ such that $B^{[t]}=V(G)$, or infinity if $B$ is not a PSD forcing set of $G$.
The \emph{PSD propagation time of  $G$}, $\ptp(G)$,  is
 \[\ptp(G)=\min\{\ptp(G;B) : |B|=\Zp(G)\}.\] 
We also define  the \emph{$k$-PSD propagation time of $G$}  to be
$ \ptp(G,k)=\min_{|B|= k}\ptp(G;B),$ so
$\ptp(G)=\ptp(G,\Zp(G)).$

\section{Single-vertex migration}
\label{svm}
For a graph $G$, we denote the set of connected components of $G$ by $\comp(G)$.  A  {\em valid initial PSD force for $S$} is a PSD force that is valid when only the vertices of $S$ have been colored blue.

\begin{observation}
\thlabel{PSDForceSwitch}
Let $G$ be a graph, $S\subset V(G)$,  $v,w\in V(G)\setminus S$, $vw \in E(G)$ and $v\neq w$.  The following are equivalent:

\begin{itemize}
\item $v\to w$ is a valid initial PSD force for $S\cup\{v\}$;
\item The removal of $vw$ from $G-S$ disconnects $v$ and $w$;
\item $vw$ is a bridge in $G-S$; and
\item $w\to v$ is a valid initial PSD force for $S\cup\{w\}$.
\end{itemize}
\end{observation}

The next result is Lemma 2.1.1 in \cite{Peters-thesis}.  We provide a shorter proof for completeness.

\begin{lemma}
\thlabel{PSDMigration}
Let $G$ be a graph, let $B$ be a PSD forcing set of $G$, and let $v\to w$ be a valid initial PSD force for $B$.  Then $B'=(B\setminus\{v\})\cup\{w\}$ is a PSD forcing set for $G$.
\end{lemma}

\begin{proof}
By \thref{PSDForceSwitch} (applied to $S=B\setminus\{v\}$), $w\to v$ is a valid initial PSD force for $B'=(B\setminus\{v\})\cup\{w\}$.  Thus, $B$ (which PSD forces $G$) is in the final coloring of $B'$, so $B'$ is a PSD forcing set for $G$.
\end{proof}

We call the process of switching $v$ and $w$ in \thref{PSDMigration} \emph{single-vertex migration}. This is illustrated in Figure \ref{singlevertexmigration}.

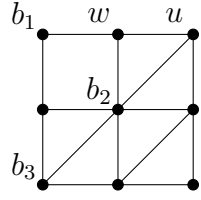
\begin{figure}[h]
    \centering
    \begin{tikzpicture}
    \filldraw[fill=black,draw=black] (1,0) circle (2pt);
    \filldraw[fill=black,draw=black] (2,0) circle (2pt);
    \filldraw[fill=black,draw=black] (1,1) circle (2pt);
    \filldraw[fill=black,draw=black] (2,1) circle (2pt);
    \filldraw[fill=black,draw=black] (1,2) circle (2pt);
    \filldraw[fill=black,draw=black] (2,2) circle (2pt);
    \filldraw[fill=black,draw=black] (3,0) circle (2pt);
    \filldraw[fill=black,draw=black] (3,1) circle (2pt);
    \filldraw[fill=black,draw=black] (3,2) circle (2pt);
    \draw (1,0) -- (3,0) -- (3,2) -- (1,2) -- (1,0) -- (3,2);
    \draw (2,0) -- (3,1);
    \draw (2,2) -- (2,0);
    \draw (1,1) -- (3,1);
    \node at (0.75,2.25) {$b_1$};
    \node at (0.75,0.25) {$b_3$};
    \node at (1.75,2.25) {$w$};
    \node at (1.75,1.25) {$b_2$};
    \node at (2.75,2.25) {$u$};
    \end{tikzpicture}
    \caption{For the graph $G$ shown above with PSD forcing set $B=\{b_1,b_2,b_3\}$, single-vertex migrations allow us to obtain several new PSD forcing sets. Since $b_1\to w$ at the first time step, one possibility is $B'=\{w,b_2,b_3\}$. A subsequent migration with $w \to u$ produces the PSD forcing set $B''=\{u,b_2,b_3\}$.}
    \label{singlevertexmigration}
\end{figure}

When starting with a PSD forcing set $B$,  a force must happen at time step $i$ within each component of $G-B^{[i-1]}$. This leads to the next observation.

\begin{observation}
\thlabel{SeparationBound} For any graph $G$ and PSD forcing set $B$,
\[\pt_+(G;B)=\max_{C\in\comp(G-B)}\pt_+(G[V(C)\cup B];B)\leq\max_{C\in\comp(G-B)}|C|.\]
\end{observation}

The next lemma exhibits a critical property of single-vertex migration that permits iterative progress towards achieving the bound \eqref{eq-ub}.  

\begin{lemma}
\thlabel{HalfwayFS}
Let $G$ be a graph of order $n$ that has a PSD forcing set $B$ of size $k$ such that $\max_{C\in\comp(G-B)}|V(C)|>\lc\frac{n-k}{2}\rc$. Then there exists a PSD forcing set $B'$ such that $|B'|=k$ and $\max_{C\in\comp(G-B')}|V(C)|<\max_{C\in\comp(G-B)}|V(C)|$.
\end{lemma}

\begin{proof}
 Let $C_0$ be the largest component of $G-B$.  Note that $|V(C_0)|\geq \frac{n-k}{2}+1$ and thus $|V(G)\setminus(V(C_0)\cup S)|\le \frac{n-k}{2}-1$.  Observe that $B$ must be able to force directly into $C_0$ (or else it could not force $G$); let $v\to w$ be a first force into $C_0$.  By single-vertex migration, the set $B'=(B\setminus\{v\})\cup\{w\}$ PSD forces $G$.  Then $|B'|=|B|=k$, and $\comp(C_0-\{w\})\subseteq\comp(G-B')$ as the removal of $B$ disconnects $C_0$ from the rest of $G$ and $N(v)\cap C_0=\{w\}$. Furthermore, $\max_{C\in\comp(G-B')}|V(C)|<|V(C_0)|=\max_{C\in\comp(G-B)}|V(C)|$ because both $\max_{C\in\comp(C_0-\{w\})}|V(C)|<|V(C_0)|$ and \[|V(G)\setminus(V(C_0)\cup B')|= |V(G)\setminus(V(C_0)\cup B)|+1\le \frac{n-k}{2}-1+1=\frac{n-k}{2}<|V(C_0)|.\qedhere\]
\end{proof}

\begin{theorem}
\thlabel{nover2}
Let $G$ be a graph of order $n$, and let $\Z_+(G)\leq k\leq n$.  Then $\pt_+(G,k)\leq\left\lceil\frac{n-k}{2}\right\rceil$.
\end{theorem}

\begin{proof}
Let $B$ be a PSD forcing set of $G$ of size $k$.    By applying \thref{HalfwayFS} repeatedly (if needed), there exists a PSD forcing set $B^\star$ of size $k$ such that $\max_{C\in\comp(G-B^\star)}|V(C)|\leq\left\lceil\frac{n-k}{2}\right\rceil$.  Then $\pt_+(G;B^\star)\leq\max_{C\in\comp(G-B^\star)}|V(C)|\leq\left\lceil\frac{n-k}{2}\right\rceil$.
\end{proof}

Observe that the bound in \thref{nover2}  is a refinement of \eqref{eq-ub}. 

\begin{corollary}\thlabel{c:nover2}
For every graph $G$  of order $n$, 
\[\pt_+(G)\leq\lc\frac{n-\Z_+(G)}{2}\rc\le \lc\frac{n-1}{2}\rc\le \frac n 2.\]
\end{corollary}

The proofs of \thref{HalfwayFS} and \thref{nover2} provide the basis for the next algorithm, which modifies a PSD forcing set $B$ to obtain  $B^\star$ such that $|B^\star|=|B|$ and $\max_{C\in\comp(G-B^\star)}|V(C)|\leq\left\lceil\frac{|V(G)|-|B^\star|}{2}\right\rceil$.

The PSD forcing set returned by this algorithm achieves the bound in \thref{nover2}.
\newpage

\begin{algorithm}[h!]
\flushleft 
\caption{ \\
\textbf{Input:} graph $G$, PSD forcing set $B$ for $G$ \\
\textbf{Output:} PSD forcing set $B^\star$ for $G$ with $|B^\star|=|B|$ and  \\$\max_{C\in \comp(G-B^\star)}|V(C)|\leq \lc \frac{|V(G)|-|B^\star|}{2} \rc$}
\label{alg1}

\begin{algorithmic}[1]
\STATE{$B^\star\coloneqq B$}
\STATE{$C_0\coloneqq$ component of $G\setminus B^\star$ with the most vertices}
\WHILE{$|V(C_0)|>\lc\frac{|V(G)|-|B^\star|}{2}\rc$}
    \STATE{$v,w\coloneqq$ a pair of vertices in $B^\star$ and $C_0$ such that $v\rightarrow w$ at the first time step}
    \STATE{$B^\star\coloneqq (B^\star\setminus \{v\})\cup \{w\}$}
    \STATE{$C_0\coloneqq$ component of $G\setminus B^\star$ with the most vertices}
\ENDWHILE
\RETURN{$B^\star$}
\end{algorithmic}
\end{algorithm}

\section{Multiple-vertex migration}
\label{mvm}

In this section we present an additional technique for modifying PSD forcing sets and use it to give an alternate proof of Theorem \ref{nover2}.   An example of the technique in  \thref{time1shift} is shown in Figure \ref{time1shiftex}. 

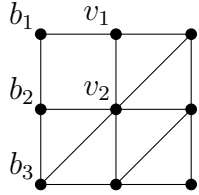
\begin{figure}[!hb]
    \centering
    \begin{tikzpicture}
    \filldraw[fill=black,draw=black] (1,0) circle (2pt);
    \filldraw[fill=black,draw=black] (2,0) circle (2pt);
    \filldraw[fill=black,draw=black] (1,1) circle (2pt);
    \filldraw[fill=black,draw=black] (2,1) circle (2pt);
    \filldraw[fill=black,draw=black] (1,2) circle (2pt);
    \filldraw[fill=black,draw=black] (2,2) circle (2pt);
    \filldraw[fill=black,draw=black] (3,0) circle (2pt);
    \filldraw[fill=black,draw=black] (3,1) circle (2pt);
    \filldraw[fill=black,draw=black] (3,2) circle (2pt);
    \draw (1,0) -- (3,0) -- (3,2) -- (1,2) -- (1,0) -- (3,2);
    \draw (2,0) -- (3,1);
    \draw (2,2) -- (2,0);
    \draw (1,1) -- (3,1);
    \node at (0.75,2.25) {$b_1$};
    \node at (0.75,1.25) {$b_2$};
    \node at (0.75,0.25) {$b_3$};
    \node at (1.75,2.25) {$v_1$};
    \node at (1.75,1.25) {$v_2$};
    \end{tikzpicture}
    \caption{For the graph $G$ shown above with PSD forcing set $B=\{b_1,b_2,b_3\}$, the method in \thref{time1shift} allows us to obtain the PSD forcing set $B'=\{v_1,v_2,b_3\}$ with $\ptp(G;B')=\ptp(G;B)-1$.}
    \label{time1shiftex}
\end{figure}

\begin{lemma}
\thlabel{time1shift}
Let $G$ be a graph with PSD forcing set $B$ such that $G-B$ is connected.  Let $B'$ be the endpoints of the PSD forcing trees after the first time step. Then $B'$ is another PSD forcing set of $G$ with $|B|=|B'|$. Furthermore, if $\pt_+(G;B)\geq 2$, then $\pt_+(G;B')=\pt_+(G;B)-1$.
\end{lemma}

\begin{proof}
Note that connectedness of $G- B$ implies that no vertex performs more than one force at the first time step. Let $B=\{b_1,b_2,\ldots,b_k\}$ and $B'=\{v_1,v_2,\ldots,v_j,b_{j+1},\ldots,b_k\}$ where $b_i\to v_i$ at the first time step for $i\le j$, and no other forces occur at the first time step. Notice that $|B'|=|B|$, and we claim that $B'$ is a PSD forcing set.

We first show that for $i\le j$, the only vertex in $B\setminus B'$ that is adjacent to $v_i$ is $b_i$.  If $v_i$ is adjacent to $b\in B$ such that $b\ne b_i$, then we have two cases:
\begin{itemize}
    \item If $b$ is not adjacent to any other vertex in $G -B$, then $b$ did not perform a force in $G$ during the first time step when using $B$ in the chosen forcing process.
    \item If $b$ is adjacent to some other vertex in $G-B$, then $b$ had more than one neighbor in $G- B$ when $B$ was selected as the PSD forcing set. Again, $b$ does not perform a force in $G$ during the first time step.
\end{itemize}
In both cases, we see that $b\in B'$ since $b$ did not perform a force during the first time step. Hence, the only vertex in $B\setminus B'$ that is adjacent to $v_i$ is $b_i$.

 Since $G- B$ is connected and $b_i$ performed a force when $B$ was chosen as a PSD forcing set, the only neighbors of $b_i\in B\setminus B'$ are other elements of $B$ and the vertex $v_i$. 
 
{This means that $\comp(G-B')$ can be partitioned into $\comp(G-(B \cup B'))$ and $\comp(G[B \setminus B'])$.  As a result, $b_i$ is the unique neighbor of $v_i$ in the component of $G- B'$ containing $b_i$.} Therefore, when $B'$ is chosen as an initial set of blue vertices, $v_i\to b_i$ at the first time step.  Since all of $B$ will be blue by the end of the first time step and $B$ is a PSD forcing set, we conclude that $B'$ is also a PSD forcing set of $G$.

Now suppose that $\pt_+(G;B)\geq 2$. Let $H=G-(B\setminus B')$. From the preceding paragraph, we observe that $B'$ begins forcing vertices in $H$ at the first time step since $B\setminus B'$ is disconnected from $H- B'$. The forcing steps in $H$ are then the same as  when $B$ was the initial PSD forcing set, but shifted by one time step. Since $\pt_+(G;B)\geq 2$, we know $\pt_+(H;B')\geq 1$, and this allows us to conclude that \[\pt_+(G;B')=\max\{1,\ptp(H;B')\}=\ptp(H;B')=\pt_+(G;B)-1. \qedhere\]
\end{proof}

\begin{remark}
In Lemma \ref{time1shift}, if we define $B''=\{v_1,v_2,\ldots,v_{j'},b_{j'+1},\ldots,b_k\}$ with $j'<j$, then the same argument shows that $B''$ is a PSD forcing set, though we cannot guarantee the second result $\ptp(G;B'')=\ptp(G;B)-1$.  Choosing $j'=1$ and combining this with the next lemma generalizes single-vertex migration.
\end{remark}

  Since PSD forcing occurs independently in the components of $G-B$, we can apply \thref{time1shift}  within the closed neighborhood of one component of $G-B$. We call this process of replacing $B$ with $B'$ within the closed neighborhood of one component (as described  in Lemma \ref{componentshift})  \emph{multiple-vertex migration}.
  The assumption in  \thref{time1shift} that $G- B$ is connected cannot be removed without such a restriction. Figure \ref{connected} illustrates both multiple-vertex migration (using one component) and a failure when moving vertices in more than one component.

\begin{figure}[h]
    \centering
    \begin{tikzpicture}
    \filldraw[fill=black,draw=black] (0,0) circle (2pt);
    \filldraw[fill=black,draw=black] (1,0) circle (2pt);
    \filldraw[fill=black,draw=black] (2,0) circle (2pt);
    \filldraw[fill=black,draw=black] (0,1) circle (2pt);
    \filldraw[fill=black,draw=black] (1,1) circle (2pt);
    \filldraw[fill=black,draw=black] (2,1) circle (2pt);
    \filldraw[fill=black,draw=black] (0,2) circle (2pt);
    \filldraw[fill=black,draw=black] (1,2) circle (2pt);
    \filldraw[fill=black,draw=black] (2,2) circle (2pt);
    \filldraw[fill=black,draw=black] (-1,0) circle (2pt);
    \filldraw[fill=black,draw=black] (-1,1) circle (2pt);
    \filldraw[fill=black,draw=black] (-1,2) circle (2pt);
    \filldraw[fill=black,draw=black] (3,0) circle (2pt);
    \filldraw[fill=black,draw=black] (3,1) circle (2pt);
    \filldraw[fill=black,draw=black] (3,2) circle (2pt);
    \draw (0,0)--(0,2)--(2,2)--(2,0)--(0,0);
    \draw (0,1)--(2,1)--(1,0)--(1,2)--(0,0)--(1,1);
    \draw (0,0)--(-1,0)--(-1,1)--(0,1);
    \draw (-1,1)--(-1,2)--(0,2);
    \draw (2,0)--(3,0)--(3,1)--(2,1);
    \draw (3,1)--(3,2)--(2,2);
    \node at (1.25,2.25) {$b_1$};
    \node at (1.25,1.25) {$b_2$};
    \node at (0.75,0.25) {$b_3$};
    \node at (2.25,2.25) {$v_1$};
    \node at (2.25,1.25) {$v_2$};
    \node at (-0.25,0.25) {$v_3$};
    \end{tikzpicture}
    \caption{Notice that $B=\{b_1,b_2,b_3\}$ is a PSD forcing set, but $\{v_1,v_2,v_3\}$ is not.  However, we can construct the PSD forcing sets $\{b_1,b_2,v_3\}$ and $\{v_1,v_2,b_3\}$.}
    \label{connected}
\end{figure}
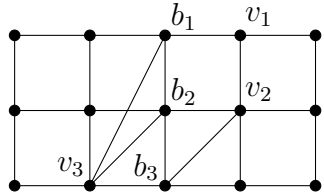

\begin{lemma}
\thlabel{componentshift}
Let $G$ be a graph with PSD forcing set $B$. Let $C$ be a connected component of $G- B$, and let $H=G[V(C)\cup B]$. If $B'$ is the set of endpoints  in $H$ of the PSD forcing trees after the first time step, then $B'$ is another PSD forcing set of $G$ with $|B'|=|B|$.
\end{lemma}

\begin{proof}
By definition of $H$, we know  $H- B$ is connected. Using Lemma \ref{time1shift}, $B'$ is a PSD forcing set for $H$ with $|B|=|B'|$. Notice that the vertices in $H- B$ are not adjacent to any vertices in $G- V(H)$. From this, we see that $B'$ will force  any white vertices in $B$  at the first step when forcing within $G$. Therefore  $B'$ will force $G$ since $B$ is a PSD forcing set of $G$. Thus $B'$ is a PSD forcing set of $G$.
\end{proof}

A PSD forcing set $B$ with $|B|=k$ is called \textit{$k$-efficient} if $\pt_+(G;B)=\pt_+(G,k)$. When $k=\Z_+(G)$,  $B$ is said to be \textit{efficient}. An application of the previous result allows us to conclude that for $k$-efficient PSD forcing sets, the two components that take the longest to force should take approximately the same time.

\begin{theorem}
\label{closeproptime}
Let $G$ be a graph and let $B$ be a PSD forcing set with $|B|=k$. Let $C_1,C_2,\ldots,C_m$ be the connected components of $G-B$, indexed so that
$\pt_+(G[V(C_i)\cup B];B)\leq \pt_+(G[V(C_{i+1})\cup B];B)$ for $i=1,2,\ldots,m-1$. If $B$ is $k$-efficient, then \[\pt_+(G[V(C_{m})\cup B];B)-\pt_+(G[V(C_{m-1})\cup B];B)\leq 1,\]
where we use the convention $C_1=\emptyset$ and $C_2=G-B$ when $G- B$ is connected.
\end{theorem}

\begin{proof}
Define $G_i=G[V(C_i)\cup B]$.  Notice that
\[\pt_+(G;B)=\max_{i=1,2,\ldots ,m}\pt_+(G_i;B)=\pt_+(G_m;B).\]
We prove the contrapositive, so suppose that $\pt_+(G_m;B)-\pt_+(G_{m-1};B)> 1$. 
Using Lemma \ref{componentshift}, if we let $B'$ be the endpoints  in $G_m$ of the PSD forcing trees  after the first time step, then $B'$ is a PSD forcing set of $G$. Additionally, since $\pt_+(G_m;B)-\pt_+(G_{m-1};B)> 1$, nonnegativity of propagation time implies $\pt_+(G_m,B)\geq 2$. The vertices of $G_m-B$ are not adjacent to any vertices in $G- G_m$, so Lemma \ref{time1shift} implies 
\[\pt_+(G_m;B')=\pt_+(G_m;B)-1.\] Since $B'$ forces $B$ at the first time step and $B$ is a PSD forcing set for $G$, we also see that the vertices in $G- V(G_m)$ will be blue by time \[\pt_+(G_{m-1};B)+1.\] Since we assumed $\pt_+(G_m;B)-\pt_+(G_{m-1};B)>1$, we see that 
\bea
        \pt_+(G;B') & =& \max\{\pt_+(G_m;B)-1,\pt_+(G_{m-1};B)+1\} \\
        & =&\pt_+(G_m;B)-1 \\
        & =&\pt_+(G;B)-1.
    \
\eea
Thus, $B$ cannot be $k$-efficient. 
\end{proof}

Theorem \ref{closeproptime} can be used to give another independent proof of Theorem \ref{nover2}.

\begin{proof}[Alternate proof of Theorem \ref{nover2}]
Let $B$ be a $k$-efficient PSD forcing set of $G$. Using the notation as in Theorem \ref{closeproptime}, the assumption that $B$ is $k$-efficient implies \[\pt_+(G_m;B)-\pt_+(G_{m-1};B)\leq 1.\] Propagation in $G_m$ and $G_{m-1}$ occur independently, so $\pt_+(G_m;B)-\pt_+(G_{m-1};B)\leq 1$ implies that at each time step, at least one force occurs in $G_m$ and at least one force occurs in $G_{m-1}$, except possibly during the last step. Since at least two forces occur at each time step with the possible exception of the last time step, we conclude that 
\[\pt_+(G,k)= \pt_+(G;B) \leq  \left\lceil \frac{|V(G)|-k}{2}\right\rceil.\qedhere\]
\end{proof}

The proof of Theorem \ref{closeproptime}  provides the basis for an algorithm for finding a PSD forcing set such that $\pt_+(G[V(C_{m})\cup B];B)-\pt_+(G[V(C_{m-1})\cup B];B)\leq 1$ holds, which we   present next. The PSD forcing set returned by this algorithm achieves the bound in Theorem \ref{nover2}, though it is not necessarily $k$-efficient.

\begin{algorithm}[h]
\flushleft 
\caption{\\
\textbf{Input:} graph $G$, PSD forcing set $B$ for $G$ \\
\textbf{Output:} PSD forcing set $B'$ for $G$ with $|B'|=|B|$ such that the two components of $G\setminus B'$ that take the longest to propagate will finish propagating within one time step of each other}
\label{alg2}
\begin{algorithmic}[1]
\STATE{$B'\coloneqq B$}
\STATE{$G_1\coloneqq$ subgraph of $G$ induced by $B'$}
\STATE{$m\coloneqq |\comp(G-B')|+1$}
\STATE{$G_2,\ldots,G_m\coloneqq$ subgraphs induced by the components of $G- B'$ combined with the vertices in $B'$, indexed so that $\ptp(G_i;B')\leq \ptp(G_{i+1};B')$ for $i=2,3\ldots,m-1$}
\WHILE{$\ptp(G_m)-\ptp(G_{m-1})\geq 2$}
    \STATE{$b_1,b_2,\ldots,b_j\coloneqq$ vertices in $B'$ that perform a force in $G_m$ at the first time step}
    \STATE{$v_1,v_2,\ldots v_j\coloneqq$ vertices of $G_m$ such that $b_i\rightarrow v_i$ at the first time step}
    \STATE{$B'\coloneqq (B'\cup \{v_1,v_2,\ldots,v_j\})\setminus \{b_1,b_2,\ldots,b_j\}$}
    \STATE{$m\coloneqq |\comp(G-B')|$}
    \STATE{$G_1,\ldots,G_m\coloneqq $ subgraphs induced by the components of $G-B'$ combined with the vertices in $B'$, indexed so that $\ptp(G_i;B')\leq \ptp(G_{i+1};B')$ for $i=1,2,\ldots,m-1$}
\ENDWHILE
\RETURN{$B'$}
\end{algorithmic}
\end{algorithm}

\newpage

\section{Consequences of the bound}\label{s:consequences}

In this section we derive several consequences of \thref{nover2}, including a tight upper bound on the PSD throttling number and tight Nordhaus-Gaddum sum bounds for PSD propagation time.  We begin by showing that the number of graphs having PSD propagation time within a fixed amount of the order is finite.

\begin{corollary}
\thlabel{finiteness}
Let $k\in \mathbb{N}$. For any graph $G$ with $|V(G)|\geq 2k+1$,
\[\pt_+(G)< |V(G)|-k.\]
The number of graphs with $\pt_+(G)\geq |V(G)|-k$ is therefore finite.
\end{corollary}
\begin{proof}
By \thref{c:nover2}, $\pt_+(G)\leq   \lc\frac{|V(G)|-1}{2}\rc.$
For any graph with $|V(G)|\geq 2k+1$, 
\[\pt_+(G)\leq \left\lceil \frac{|V(G)|-1}{2} \right\rceil< |V(G)|-k.\]
Since there are only finitely many graphs with $|V(G)|\leq 2k$, this implies that the number of graphs with $\pt_+(G)\geq |V(G)|-k$ is finite.
\end{proof}

Warnberg \cite{warnberg} characterized the graphs achieving $\pt_+(G)\geq |V(G)|-k$ for $k=1,2$, thereby establishing the results in Corollary \ref{finiteness} for $k=1,2$.
\begin{theorem}\thlabel{warnberg-small}\cite{warnberg}
Let $G$ be a graph.
\begin{enumerate}
\item   $\ptp(G)=|V(G)|-1$ if and only if $G=P_2$.
\item   $\ptp(G)=|V(G)|-2$ if and only if $G$ is one of $P_3$, $P_4$, $C_3$, or $P_2\dcup P_1$, where $G\dcup H$ denotes the disjoint union of $G$ and $H$.  
\end{enumerate}
\end{theorem}  

As the value of $k$ increases, the number of graphs $G$ such that $\pt_+(G) = |V(G)|-k$ grows rapidly. The graphs having $\ptp(G)=|V(G)|-3$ and  $\ptp(G)=|V(G)|-4$ are shown in Figures \ref{f:nminus3} and \ref{f:nminus4}, respectively.

\begin{figure}[!h]
\tikzsetnextfilename{figure_nminus3}
\begin{tikzpicture}[set graph scales={1.17}{1.44}{1.44}]\DrawGraphs{n minus 3}{5x4}\end{tikzpicture}
\caption{Graphs with $\pt_+(G)=|V(G)|-3$\label{f:nminus3}.}
\end{figure}
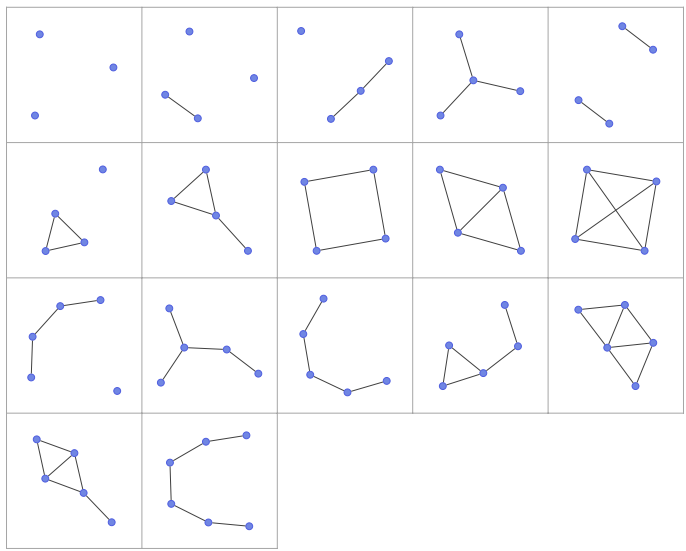

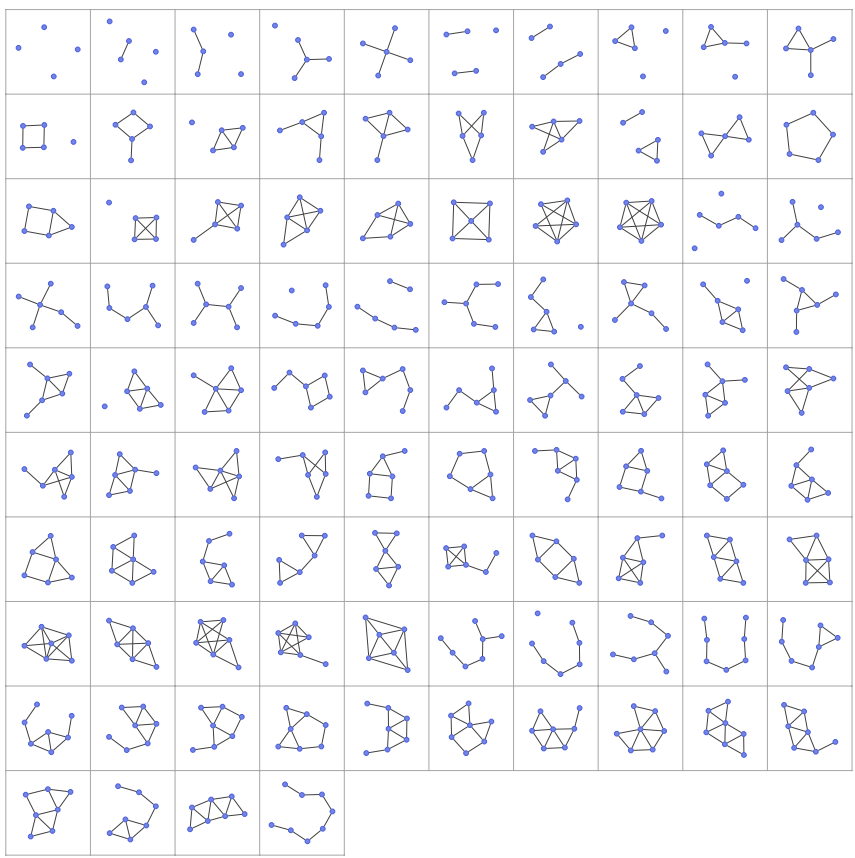
\begin{figure}[!h]
\tikzsetnextfilename{figure_nminus4}
\begin{tikzpicture}[set graph scales={1}{1}{0.9}]\DrawGraphs{n minus 4}{10x10}\end{tikzpicture}
\caption{Graphs with $\pt_+(G)=|V(G)|-4$\label{f:nminus4}.}
\end{figure}

Throttling  minimizes the sum  of the resources used to accomplish a task (number of blue vertices) and the time needed to complete that task (propagation time).  
Theorem \ref{nover2} yields a bound on the \emph{PSD throttling number} of a graph $G$ of order $n$, which is defined to be $\thp(G)=\min_{\Zp(G)\le k\le n}(k+\ptp(G,k))$.
\begin{corollary}
For any graph $G$ on $n$ vertices,
\[\thp(G)\leq 
\left\lceil \frac{n+\Z_+(G)}{2}\right\rceil,\]
and this bound is tight for $K_n$ for all $n$.
\end{corollary}

\begin{proof}
By    the definition of throttling and \thref{nover2}, $\thp(G)\le k+\ptp(G,k)\le k+\lc\frac{n-k}{2}\rc=
\lc \frac{n+k}{2}\rc$ for $k=\Zp(G),\dots,n$.
Thus $\thp(G)\le \lc \frac{n+\Z_+(G)}{2}\rc$. 
For tightness,  $\Zp(K_n)=n-1$ and $\ptp(K_n)=1$, so $\thp(K_n)=n=\lc\frac{n+(n-1)}2\rc$. 
\end{proof}

Given a graph $G$, its \emph{complement} $\ol{G}$ is the graph with vertex set $V(G)$ and edge set
\[E \left(\ol{G} \right) = \left\{ uv : u,v \in V(G) \text{ distinct and } uv \not \in E(G) \right \} .\]
 The \emph{Nordhaus-Gaddum sum problem} for a graph parameter $\zeta$ is to determine a  lower or upper bound on $\zeta(G)+\zeta(\ol{G})$ that is tight for graphs of arbitrarily large order. We can use Theorem \ref{nover2} to give a tight Nordhaus-Gaddum sum upper bound for the PSD propagation time of a graph and its complement.   We recall the next (tight) Nordhaus-Gaddum sum bounds for the PSD zero forcing number.

\begin{theorem}\cite{ekstrand}\label{ZpNG}
Let $G$ be a graph of order $n\ge 2$.  Then  $n-2 \leq \Zp(G) + \Zp(\ol{G}) \leq 2n-1$,  and both bounds are tight for arbitrarily large $n$.
\end{theorem}

\begin{theorem}\label{t:NG}
Let $G$ be a graph of order $n \geq 2$.  Then \[1 \leq \ptp(G) + \ptp(\overline{G}) \leq \frac{n}{2} + 2.\]  The lower bound is tight for every $n\ge 2$ and the upper bound is tight for every even $n\ge 8$.
\end{theorem}

\begin{proof}
  Since $n$ is at least 2, either the graph or its complement has an edge. Therefore either $\ptp(G) \geq 1$ or $\ptp(\ol{G})\geq 1$.  Observe that the lower bound is achieved by the complete graph $K_n$ for $n\geq 2$.

To establish the upper bound:
\bea
    \ptp(G) + \ptp(\ol{G}) &\leq& \lc \frac{n-\Zp(G)}{2} \rc + \lc \frac{n-\Zp(\ol{G})}{2} \rc
    \\     
    &\leq &\frac{n-\Zp(G)}{2} + \frac{n-\Zp(\ol{G})}{2} +1 
    \\
    &\leq & n+1 - \frac{\Zp(G) + \Zp(\ol{G})}{2} \\
    &\leq & n+1 - \frac{n-2}{2} 
    \\
    &\leq & \frac{n}{2} + 2.
\eea

To establish tightness for even $n\geq 8$, consider the graph $H_{2k+8}$  on $2k+8$ vertices (with $k\geq 0$) shown in Figure \ref{fig:H8exp}. 

\begin{figure}[h!]
    \centering
    \begin{tikzpicture}
    \filldraw[fill=black,draw=black] (0,0) circle (2pt);
    \filldraw[fill=black,draw=black] (1,-0.5) circle (2pt);
    \filldraw[fill=black,draw=black] (2,-1) circle (2pt);
    \filldraw[fill=black,draw=black] (1,-1.5) circle (2pt);
    \filldraw[fill=black,draw=black] (0,-2) circle (2pt);
    \filldraw[fill=black,draw=black] (-1,-1.5) circle (2pt);
    \filldraw[fill=black,draw=black] (-2,-1) circle (2pt);
    \filldraw[fill=black,draw=black] (-1,-0.5) circle (2pt);
    \filldraw[fill=black,draw=black] (3,-1) circle (2pt);
    \filldraw[fill=black,draw=black] (-3,-1) circle (2pt);
    \filldraw[fill=black,draw=black] (5,-1) circle (2pt);
    \filldraw[fill=black,draw=black] (-5,-1) circle (2pt);
    \draw (0,0) -- (2,-1) -- (0,-2) -- (-2,-1) -- (0,0) -- (0,-2);
    \draw (-1,-1.5) -- (-1,-0.5) -- (1,-0.5) -- (-1,-1.5) -- (1,-1.5) -- (1,-0.5);
    \node at (4,-1) {$\ldots$};
    \node at (-4,-1) {$\ldots$};
    \draw (2,-1) -- (3.5,-1);
    \draw (-2,-1) -- (-3.5,-1);
    \draw (4.5,-1) -- (5,-1);
    \draw (-4.5,-1) -- (-5,-1);
    \node at (2,-1.375) {$b_0$};
    \node at (-2,-1.375) {$a_0$};
    \node at (3,-1.375) {$b_1$};
    \node at (-3,-1.375) {$a_1$};
    \node at (5,-1.375) {$b_k$};
    \node at (-5,-1.375) {$a_k$};
    \node at (-1.25,-0.25) {$x$};
    \node at (-1.25,-1.75) {$y'$};
    \node at (1.25,-0.25) {$y$};
    \node at (1.25,-1.75) {$x'$};
    \node at (0,0.25) {$z$};
    \node at (0,-2.25) {$z'$};
    \end{tikzpicture}
\caption{The graph $H_{2k+8}$, which has  order $2k+8$ and  $\ptp(H_{2k+8})+\ptp(\ol{H_{2k+8}})=(k+4)+2$.
\label{fig:H8exp}}  
\end{figure}

It is straightforward to verify the following properties of $H_8$:  $\Zp(H_8)=\pt_+(H_8)=3$. For any minimum PSD forcing set $B$ of $H_8$, $a_0\not\in B$ or $b_0\not\in B$, and one of  $a_0$ or $b_0$ is the last vertex forced.
Since $\overline{H_8}\cong H_8$, \[\ptp(H_8)+\ptp(\ol{H_8})=6=\frac{8}{2}+2,\] so $H_8$ gives a tight bound for $n=8$. 

Now assume $k\ge 1$.  By results in \cite{ekstrand}, $\Zp(H_{2k+8})=3$ and any minimum PSD forcing set for $H_{2k+8}$ must contain at least two vertices in $V(H_8)\setminus \{a_0,b_0\}$. If $B$ is a PSD forcing set with two vertices in $V(H_8)\setminus \{a_0,b_0\}$ and  a third vertex $a_i$ from $\{a_1,a_2,\ldots,a_k\}$, then migration from $a_i$ to $a_{i-1}$ produces a PSD forcing set $B'=(B\setminus \{a_{i}\})\cup \{a_{i-1}\}$ with $\ptp(H_{2k+8};B')=\ptp(H_{2k+8};B)-1$, as $b_0$ will be the last vertex in $H_8$ forced regardless of what $B$ and $B'$ are. A similar argument applies when $B$ contains some vertex in $\{b_1,b_2,\ldots,b_k\}$, and in these situations, we conclude that $B$ cannot be efficient. From this, we see that if we wish to select an efficient PSD forcing set $B$ for $H_{2k+8}$, we must select three vertices from $H_8$. Regardless of which three vertices we select, either $a_0$ or $b_0$ will be the last vertex of $H_8$ forced, implying that
\[\ptp(H_{2k+8})=\ptp(H_{2k+8};B)=\ptp(H_8)+k=k+3.\]

Since the order of $\ol{H_{2k+8}}$ is $2k+8$ and $\ptp(H_{2k+8})=k+3$, to complete the proof it suffices to show that $\ptp(\ol{H_{2k+8}})=3$. The software \cite{sage-PSDprop} was used to verify that $\ptp(\ol{H_{10}})=\ptp(\ol{H_{12}})=3$, so we focus on the remaining cases. Fix $k\ge 3$, and let $H=H_{2k+8}$, so $\ol{H}=\ol{H_{2k+8}}$. 

We first show $\Zp(\ol{H})= 2k+3$ and $\ptp(\ol{H})\le 3$. Notice that Theorem \ref{ZpNG} and  $\Zp(H)=3$ imply that $\Zp(\ol{H})\ge 2k+3$. Since $\ol{H}$ contains $\ol{H_8}$ as a subgraph,  we let $B=B_0\cup X$ where $B_0$ is an efficient PSD forcing set for $\ol{H_8}\cong H_8$ and $X=\{a_1,\ldots,a_k,b_1,\ldots,b_k\}$. Notice that $|B|=2k+3$. 
The graph $\ol{H}$ also contains $\ol{H_{12}}$ as a subgraph, and whenever all vertices of $X$ are blue, we may assume forcing takes place entirely within $\ol{H_{12}}$, since $a_i\to w$ can be replaced by $a_2\to w$ for $i\ge 3$, and similarly for $b_i$. Combined, we conclude that $B$ is a PSD forcing set for $\ol{H}$, $\Zp(\ol{H})=2k+3$, and $\ptp(\ol{H})\le \ptp(\ol{H};B)=\ptp(\ol{H_{12}};B\cap V(\ol{H_{12}}))=3$.

To prove $\ptp(\ol{H})\geq  3$, we show that an efficient PSD forcing set $B$ must contain all vertices of $X$. Let $B\subset V(\ol H)$ be a PSD forcing set of $\ol H$ such that  $|B|=2k+3$ and $X\not\subseteq B$. Let $W=V(\ol H)\setminus B$, so $W\cap X\ne\emptyset$. 

We begin by showing  $\ol{H}[W]$ is connected.  Suppose first that $|W\cap X|\ge 2$, and without loss of generality, assume there is some $a_i\in W\cap X$.  
Then $\ol{H}-B=\ol H[W]$ is connected because every vertex except $a_{i-1}$ and $a_{i+1}$ is adjacent to $a_i$, $|W\cap X|\ge 2$, and $|W|=5$. 
Alternatively, suppose $|W\cap X|=1$, and assume without loss of generality that $W\cap X=\{a_i\}$. If $a_0\notin W$, then $a_i$ is adjacent to all other vertices in $W$, again implying $\ol{H}[W]$ is connected. If $a_0\in W$, then there must be three white vertices in $\ol{H_8}\setminus \{a_0\}$, which are all adjacent to $a_i$. Furthermore, $\deg_{\ol{H_8}}a_0=5$ and $|W\cap V(\ol{H_8})|=4$ imply that some neighbor of $a_0$ is white. In all cases, $\ol{H}[W]$ is connected.

Since $\ol{H}-B$ is connected, the first force must be a standard force. If vertex $u$ performs the first force, then $\deg_{\ol{H}}u\le \Zp(\ol H)=  2k+3$, so $\deg_H u\ge 4$, i.e., $u\in \{x,x',y,y'\}$. If $|W\cap X|\geq 2$, then $u$ is adjacent to multiple white vertices,  implying $u$ cannot perform a force, and $B$ would not be a PSD forcing set. If $|W\cap X|=1$, notice that $a_i\in W\cap X$ is adjacent to all of $x,x',y$ and $y'$. Then the only vertex forced at the first time step is $a_i$, implying $\ptp(\ol{H};B)\geq 2$. 
So Lemma \ref{time1shift} implies that $B$ is not efficient. 

Thus, any efficient PSD forcing set $B$ for $\ol{H}$ must contain all of $X$. For any such $B$, we can again assume forcing takes place entirely within $\ol{H_{12}}$, and we conclude that $\ptp(\ol{H})=\ptp(\ol{H};B)= \ptp(\ol{H_{12}};B\cap V(\ol{H_{12}}))\geq 3$. 
 \end{proof}

In order for the upper bound in Theorem \ref{t:NG} to be tight, $n$ must be even. In addition to the family $H_{2k+8}$ presented in the proof, the upper bound is tight for   $P_4$.  A computer search shows there is no graph $G$ of order 6 realizing $\ptp(G) + \ptp(\overline{G}) =5=\frac{6}{2} + 2$.

Finally we consider the maximum PSD propagation time for graphs with arbitrary order and fixed PSD forcing number. For a fixed positive integer $k$ and  $n>k$, define \[\zeta(n,k)=\max\{\pt_+(G): \, |V(G)|=n \text{ and } \Z_+(G)=k\}.\]
By \thref{nover2}, $\zeta(n,k)\le \lc \frac{n-k}{2}\rc$. We construct a family of examples realizing this bound.
Define the lollipop graph $L_{m,r}$ as the graph obtained by starting with the complete graph $K_m$ with $m\ge 3$ and a (disjoint) path $P_r$, and then adding  an edge between some vertex $v$ of $K_m$ and an endpoint of $P_r$; the order of $L_{m,r}$ is $m+r$.  See Figure \ref{lollipop}. 

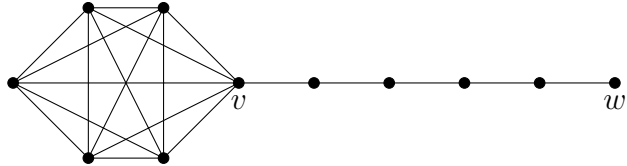
\begin{figure}[h!]
    \centering
    \begin{tikzpicture}
    \filldraw[fill=black,draw=black] (0,0) circle (2pt);
    \filldraw[fill=black,draw=black] (3,0) circle (2pt);
    \filldraw[fill=black,draw=black] (4,0) circle (2pt);
    \filldraw[fill=black,draw=black] (5,0) circle (2pt);
    \filldraw[fill=black,draw=black] (6,0) circle (2pt);
    \filldraw[fill=black,draw=black] (7,0) circle (2pt);
    \filldraw[fill=black,draw=black] (8,0) circle (2pt);
    \filldraw[fill=black,draw=black] (1,1) circle (2pt);
    \filldraw[fill=black,draw=black] (2,1) circle (2pt);
    \filldraw[fill=black,draw=black] (1,-1) circle (2pt);
    \filldraw[fill=black,draw=black] (2,-1) circle (2pt);
    \draw (8,0) -- (0,0) -- (1,1) -- (2,1) -- (3,0) -- (2,-1) -- (1,-1) -- (0,0) -- (2,1) -- (2,-1) -- (0,0);
    \draw (1,1) -- (3,0) -- (1,-1) -- (1,1) -- (2,-1);
    \draw (1,-1) -- (2,1);
    \node at (3,-0.25) {$v$};
    \node at (8,-0.25) {$w$};
    \end{tikzpicture}
    \caption{The lollipop graph $L_{6,5}$\label{lollipop}.}
\end{figure}

\begin{proposition}
\label{nover2tight}  For $m\ge 3$ and $r\ge 1$, \[\ptp(L_{m,r})=\lc \frac{|V(L_{m,r})|-\Zp(L_{m,r})}{2}\rc.\]
\end{proposition}

\begin{proof} 
 Let $v$ be the vertex of degree $m$ and let $w$ be the vertex of degree one in $L_{m,r}$.  It is clear that $\Z_+(L_{m,r})=m-1$ since $L_{m,r}$ contains $K_m$ as a subgraph and any set of $m-1$ of the vertices in $K_m$ is a PSD forcing set.

Consider a PSD forcing set $B$ consisting of $m-2$ of the vertices in $K_m \setminus \{v\}$ and the vertex of $P_r$ at distance $\lc \frac r 2\rc$ from $w$,
It will take $\lc \frac r 2\rc$ time steps to force the vertices of $P_r$ and  $r-\lc \frac r 2 \rc+1=\lc \frac {r+1} 2\rc$ time steps to force the last vertex of $K_m$.
Thus,  \[\pt_+(L_{m,r};B)=\lc \frac {r+1} 2\rc=\lc \frac {(m+r)-(m-1)} 2\rc = \left\lceil \frac{|V(L_{m,r})|-\Z_+(L_{m,r})}{2}\right\rceil.\]

The set $B$ is efficient for $L_{m,r}$ because any PSD forcing set must contain at least $m-2$ of the vertices in $K_m$ and any other choice for the last vertex results in a propagation time that is at least as large.
\end{proof}

\begin{corollary}
For any $n\ge k\ge 1$, there exists a graph $G$ such that $|V(G)|=n$, $\Z_+(G)=k$, and $\pt_+(G)= \lc \frac{n-k}{2}\rc$. Thus, the bound 
$\pt_+(G)\leq \lc \frac{|V(G)|-\Z_+(G)}{2} \rc$
 is tight for each $\Zp(G)$.   
\end{corollary}

\begin{corollary}
For a fixed positive integer $k$,
\[\lim_{n\to\infty} \frac{\zeta(n,k)}{n}=\frac{1}{2}.\]
\end{corollary}

\begin{proof}
Starting with Theorem \ref{nover2} and letting $n\to\infty$ implies
\[\lim_{n\to\infty} \frac{\zeta(n,k)}{n}\leq \frac{1}{2}.\] For the lower bound with fixed $k$, Proposition \ref{nover2tight} implies that for any $n\ge k+3$, 
$\pt_+(L_{k+1,n-k-1})=\lc \frac{n-k}{2}\rc.$
Then \[\frac{\zeta(n,k)}{n}\geq \frac{\lceil \frac{n-k}{2}\rceil}{n}\geq \frac{n-k}{2n},\]
and letting $n\to\infty$ implies the result.
\end{proof}

\section*{Acknowledgements}

The research of all the authors was partially supported by NSF grant 1916439.  The research of Yaqi Zhang was also partially supported by Simons Foundation grant 355645 and NSF grant 2000037.

\bibliographystyle{plain}

\end{document}

%% file: nminus3.tex
\tikzset{/graphs en masse/add graph={
  \node (v1) at (0.862,0.569) {};
  \node (v2) at (0.138,0.125) {};
  \node (v3) at (0.182,0.875) {};
}}
\tikzset{/graphs en masse/add graph={
  \node (v1) at (0.09,0.317) {};
  \node (v2) at (0.91,0.471) {};
  \node (v3) at (0.314,0.901) {};
  \node (v4) at (0.391,0.099) {} edge (v1);
}}
\tikzset{/graphs en masse/add graph={
  \node (v1) at (0.37,0.094) {};
  \node (v2) at (0.905,0.628) {};
  \node (v3) at (0.095,0.906) {};
  \node (v4) at (0.645,0.353) {} edge (v1) edge (v2);
}}
\tikzset{/graphs en masse/add graph={
  \node (v1) at (0.132,0.125) {};
  \node (v2) at (0.305,0.875) {};
  \node (v3) at (0.868,0.35) {};
  \node (v4) at (0.435,0.45) {} edge (v1) edge (v2) edge (v3);
}}
\tikzset{/graphs en masse/add graph={
  \node (v1) at (0.44,0.05) {};
  \node (v2) at (0.56,0.95) {};
  \node (v3) at (0.156,0.267) {} edge (v1);
  \node (v4) at (0.844,0.733) {} edge (v2);
}}
\tikzset{/graphs en masse/add graph={
  \node (v1) at (0.594,0.203) {};
  \node (v2) at (0.764,0.877) {};
  \node (v3) at (0.236,0.123) {} edge (v1);
  \node (v4) at (0.325,0.468) {} edge (v1) edge (v3);
}}
\tikzset{/graphs en masse/add graph={
  \node (v1) at (0.146,0.585) {};
  \node (v2) at (0.854,0.125) {};
  \node (v3) at (0.466,0.875) {} edge (v1);
  \node (v4) at (0.559,0.451) {} edge (v1) edge (v2) edge (v3);
}}
\tikzset{/graphs en masse/add graph={
  \node (v1) at (0.125,0.762) {};
  \node (v2) at (0.875,0.238) {};
  \node (v3) at (0.762,0.875) {} edge (v1) edge (v2);
  \node (v4) at (0.238,0.125) {} edge (v1) edge (v2);
}}
\tikzset{/graphs en masse/add graph={
  \node (v1) at (0.874,0.125) {};
  \node (v2) at (0.126,0.875) {};
  \node (v3) at (0.292,0.292) {} edge (v1) edge (v2);
  \node (v4) at (0.708,0.708) {} edge (v1) edge (v2) edge (v3);
}}
\tikzset{/graphs en masse/add graph={
  \node (v1) at (0.766,0.125) {};
  \node (v2) at (0.875,0.766) {} edge (v1);
  \node (v3) at (0.125,0.234) {} edge (v1) edge (v2);
  \node (v4) at (0.234,0.875) {} edge (v1) edge (v2) edge (v3);
}}
\tikzset{/graphs en masse/add graph={
  \node (v1) at (0.371,0.864) {};
  \node (v2) at (0.103,0.205) {};
  \node (v3) at (0.897,0.08) {};
  \node (v4) at (0.743,0.92) {} edge (v1);
  \node (v5) at (0.116,0.581) {} edge (v1) edge (v2);
}}
\tikzset{/graphs en masse/add graph={
  \node (v1) at (0.658,0.463) {};
  \node (v2) at (0.128,0.843) {};
  \node (v3) at (0.05,0.157) {};
  \node (v4) at (0.95,0.24) {} edge (v1);
  \node (v5) at (0.266,0.481) {} edge (v1) edge (v2) edge (v3);
}}
\tikzset{/graphs en masse/add graph={
  \node (v1) at (0.115,0.606) {};
  \node (v2) at (0.523,0.067) {};
  \node (v3) at (0.302,0.933) {} edge (v1);
  \node (v4) at (0.885,0.173) {} edge (v2);
  \node (v5) at (0.179,0.23) {} edge (v1) edge (v2);
}}
\tikzset{/graphs en masse/add graph={
  \node (v1) at (0.153,0.125) {};
  \node (v2) at (0.847,0.493) {};
  \node (v3) at (0.211,0.501) {} edge (v1);
  \node (v4) at (0.725,0.875) {} edge (v2);
  \node (v5) at (0.527,0.245) {} edge (v1) edge (v2) edge (v3);
}}
\tikzset{/graphs en masse/add graph={
  \node (v1) at (0.584,0.875) {};
  \node (v2) at (0.682,0.125) {};
  \node (v3) at (0.154,0.831) {} edge (v1);
  \node (v4) at (0.846,0.525) {} edge (v1) edge (v2);
  \node (v5) at (0.421,0.48) {} edge (v1) edge (v2) edge (v3) edge (v4);
}}
\tikzset{/graphs en masse/add graph={
  \node (v1) at (0.233,0.52) {};
  \node (v2) at (0.846,0.117) {};
  \node (v3) at (0.502,0.756) {} edge (v1);
  \node (v4) at (0.154,0.883) {} edge (v1) edge (v3);
  \node (v5) at (0.587,0.388) {} edge (v1) edge (v2) edge (v3);
}}
\tikzset{/graphs en masse/add graph={
  \node (v1) at (0.466,0.861) {};
  \node (v2) at (0.146,0.287) {};
  \node (v3) at (0.866,0.08) {};
  \node (v4) at (0.84,0.92) {} edge (v1);
  \node (v5) at (0.489,0.115) {} edge (v2) edge (v3);
  \node (v6) at (0.134,0.668) {} edge (v1) edge (v2);
}}

%% file: nminus4.tex
\tikzset{/graphs en masse/add graph={
  \node (v1) at (0.442,0.864) {};
  \node (v2) at (0.063,0.559) {};
  \node (v3) at (0.937,0.536) {};
  \node (v4) at (0.583,0.136) {};
}}
\tikzset{/graphs en masse/add graph={
  \node (v1) at (0.324,0.387) {};
  \node (v2) at (0.841,0.501) {};
  \node (v3) at (0.67,0.05) {};
  \node (v4) at (0.159,0.95) {};
  \node (v5) at (0.444,0.66) {} edge (v1);
}}
\tikzset{/graphs en masse/add graph={
  \node (v1) at (0.15,0.83) {};
  \node (v2) at (0.21,0.17) {};
  \node (v3) at (0.702,0.754) {};
  \node (v4) at (0.85,0.171) {};
  \node (v5) at (0.292,0.508) {} edge (v1) edge (v2);
}}
\tikzset{/graphs en masse/add graph={
  \node (v1) at (0.441,0.675) {};
  \node (v2) at (0.391,0.112) {};
  \node (v3) at (0.899,0.395) {};
  \node (v4) at (0.101,0.888) {};
  \node (v5) at (0.572,0.385) {} edge (v1) edge (v2) edge (v3);
}}
\tikzset{/graphs en masse/add graph={
  \node (v1) at (0.624,0.85) {};
  \node (v2) at (0.376,0.15) {};
  \node (v3) at (0.85,0.376) {};
  \node (v4) at (0.15,0.624) {};
  \node (v5) at (0.5,0.5) {} edge (v1) edge (v2) edge (v3) edge (v4);
}}
\tikzset{/graphs en masse/add graph={
  \node (v1) at (0.445,0.804) {};
  \node (v2) at (0.255,0.182) {};
  \node (v3) at (0.867,0.818) {};
  \node (v4) at (0.133,0.756) {} edge (v1);
  \node (v5) at (0.573,0.219) {} edge (v2);
}}
\tikzset{/graphs en masse/add graph={
  \node (v1) at (0.413,0.874) {};
  \node (v2) at (0.312,0.126) {};
  \node (v3) at (0.859,0.468) {};
  \node (v4) at (0.141,0.704) {} edge (v1);
  \node (v5) at (0.571,0.321) {} edge (v2) edge (v3);
}}
\tikzset{/graphs en masse/add graph={
  \node (v1) at (0.367,0.863) {};
  \node (v2) at (0.874,0.808) {};
  \node (v3) at (0.537,0.137) {};
  \node (v4) at (0.126,0.659) {} edge (v1);
  \node (v5) at (0.417,0.554) {} edge (v1) edge (v4);
}}
\tikzset{/graphs en masse/add graph={
  \node (v1) at (0.292,0.868) {};
  \node (v2) at (0.816,0.624) {};
  \node (v3) at (0.648,0.132) {};
  \node (v4) at (0.184,0.569) {} edge (v1);
  \node (v5) at (0.496,0.632) {} edge (v1) edge (v2) edge (v4);
}}
\tikzset{/graphs en masse/add graph={
  \node (v1) at (0.334,0.845) {};
  \node (v2) at (0.85,0.689) {};
  \node (v3) at (0.516,0.155) {};
  \node (v4) at (0.15,0.55) {} edge (v1);
  \node (v5) at (0.515,0.527) {} edge (v1) edge (v2) edge (v3) edge (v4);
}}
\tikzset{/graphs en masse/add graph={
  \node (v1) at (0.129,0.655) {};
  \node (v2) at (0.437,0.339) {};
  \node (v3) at (0.875,0.418) {};
  \node (v4) at (0.125,0.331) {} edge (v1) edge (v2);
  \node (v5) at (0.444,0.669) {} edge (v1) edge (v2);
}}
\tikzset{/graphs en masse/add graph={
  \node (v1) at (0.755,0.647) {};
  \node (v2) at (0.245,0.672) {};
  \node (v3) at (0.475,0.146) {};
  \node (v4) at (0.51,0.854) {} edge (v1) edge (v2);
  \node (v5) at (0.49,0.462) {} edge (v1) edge (v2) edge (v3);
}}
\tikzset{/graphs en masse/add graph={
  \node (v1) at (0.875,0.627) {};
  \node (v2) at (0.436,0.293) {};
  \node (v3) at (0.125,0.707) {};
  \node (v4) at (0.564,0.591) {} edge (v1) edge (v2);
  \node (v5) at (0.744,0.338) {} edge (v1) edge (v2) edge (v4);
}}
\tikzset{/graphs en masse/add graph={
  \node (v1) at (0.826,0.85) {};
  \node (v2) at (0.758,0.15) {};
  \node (v3) at (0.174,0.588) {};
  \node (v4) at (0.784,0.502) {} edge (v1) edge (v2);
  \node (v5) at (0.504,0.712) {} edge (v1) edge (v3) edge (v4);
}}
\tikzset{/graphs en masse/add graph={
  \node (v1) at (0.811,0.602) {};
  \node (v2) at (0.189,0.762) {};
  \node (v3) at (0.363,0.15) {};
  \node (v4) at (0.543,0.85) {} edge (v1) edge (v2);
  \node (v5) at (0.454,0.505) {} edge (v1) edge (v2) edge (v3) edge (v4);
}}
\tikzset{/graphs en masse/add graph={
  \node (v1) at (0.52,0.15) {};
  \node (v2) at (0.314,0.839) {};
  \node (v3) at (0.686,0.85) {};
  \node (v4) at (0.372,0.507) {} edge (v1) edge (v2) edge (v3);
  \node (v5) at (0.647,0.515) {} edge (v1) edge (v2) edge (v3);
}}
\tikzset{/graphs en masse/add graph={
  \node (v1) at (0.85,0.728) {};
  \node (v2) at (0.15,0.644) {};
  \node (v3) at (0.312,0.272) {};
  \node (v4) at (0.467,0.721) {} edge (v1) edge (v2) edge (v3);
  \node (v5) at (0.584,0.452) {} edge (v1) edge (v2) edge (v3) edge (v4);
}}
\tikzset{/graphs en masse/add graph={
  \node (v1) at (0.757,0.449) {};
  \node (v2) at (0.524,0.86) {};
  \node (v3) at (0.472,0.29) {} edge (v1);
  \node (v4) at (0.243,0.702) {} edge (v2);
  \node (v5) at (0.743,0.14) {} edge (v1) edge (v3);
}}
\tikzset{/graphs en masse/add graph={
  \node (v1) at (0.714,0.784) {};
  \node (v2) at (0.15,0.545) {};
  \node (v3) at (0.85,0.451) {} edge (v1);
  \node (v4) at (0.294,0.216) {} edge (v2);
  \node (v5) at (0.5,0.503) {} edge (v1) edge (v2) edge (v3) edge (v4);
}}
\tikzset{/graphs en masse/add graph={
  \node (v1) at (0.628,0.15) {};
  \node (v2) at (0.156,0.672) {};
  \node (v3) at (0.844,0.527) {} edge (v1);
  \node (v4) at (0.202,0.24) {} edge (v1) edge (v2);
  \node (v5) at (0.553,0.85) {} edge (v2) edge (v3);
}}
\tikzset{/graphs en masse/add graph={
  \node (v1) at (0.578,0.651) {};
  \node (v2) at (0.15,0.351) {};
  \node (v3) at (0.85,0.411) {} edge (v1);
  \node (v4) at (0.217,0.715) {} edge (v1) edge (v2);
  \node (v5) at (0.511,0.285) {} edge (v1) edge (v2) edge (v3);
}}
\tikzset{/graphs en masse/add graph={
  \node (v1) at (0.85,0.551) {};
  \node (v2) at (0.15,0.772) {};
  \node (v3) at (0.527,0.228) {} edge (v1);
  \node (v4) at (0.542,0.542) {} edge (v1) edge (v3);
  \node (v5) at (0.844,0.233) {} edge (v1) edge (v3) edge (v4);
}}
\tikzset{/graphs en masse/add graph={
  \node (v1) at (0.508,0.78) {};
  \node (v2) at (0.15,0.22) {};
  \node (v3) at (0.85,0.727) {} edge (v1);
  \node (v4) at (0.794,0.386) {} edge (v1) edge (v3);
  \node (v5) at (0.466,0.449) {} edge (v1) edge (v2) edge (v3) edge (v4);
}}
\tikzset{/graphs en masse/add graph={
  \node (v1) at (0.467,0.85) {};
  \node (v2) at (0.228,0.15) {};
  \node (v3) at (0.772,0.651) {} edge (v1);
  \node (v4) at (0.575,0.362) {} edge (v1) edge (v2) edge (v3);
  \node (v5) at (0.282,0.553) {} edge (v1) edge (v2) edge (v3) edge (v4);
}}
\tikzset{/graphs en masse/add graph={
  \node (v1) at (0.365,0.588) {};
  \node (v2) at (0.556,0.268) {};
  \node (v3) at (0.672,0.757) {} edge (v1) edge (v2);
  \node (v4) at (0.15,0.243) {} edge (v1) edge (v2);
  \node (v5) at (0.85,0.457) {} edge (v1) edge (v2) edge (v3);
}}
\tikzset{/graphs en masse/add graph={
  \node (v1) at (0.241,0.776) {};
  \node (v2) at (0.759,0.224) {};
  \node (v3) at (0.776,0.759) {} edge (v1) edge (v2);
  \node (v4) at (0.224,0.241) {} edge (v1) edge (v2);
  \node (v5) at (0.5,0.5) {} edge (v1) edge (v2) edge (v3) edge (v4);
}}
\tikzset{/graphs en masse/add graph={
  \node (v1) at (0.28,0.748) {};
  \node (v2) at (0.201,0.426) {};
  \node (v3) at (0.67,0.804) {} edge (v1) edge (v2);
  \node (v4) at (0.52,0.196) {} edge (v1) edge (v2) edge (v3);
  \node (v5) at (0.799,0.45) {} edge (v1) edge (v2) edge (v3) edge (v4);
}}
\tikzset{/graphs en masse/add graph={
  \node (v1) at (0.29,0.771) {};
  \node (v2) at (0.196,0.407) {} edge (v1);
  \node (v3) at (0.665,0.795) {} edge (v1) edge (v2);
  \node (v4) at (0.514,0.205) {} edge (v1) edge (v2) edge (v3);
  \node (v5) at (0.804,0.445) {} edge (v1) edge (v2) edge (v3) edge (v4);
}}
\tikzset{/graphs en masse/add graph={
  \node (v1) at (0.696,0.561) {};
  \node (v2) at (0.125,0.59) {};
  \node (v3) at (0.05,0.096) {};
  \node (v4) at (0.445,0.904) {};
  \node (v5) at (0.95,0.393) {} edge (v1);
  \node (v6) at (0.407,0.429) {} edge (v1) edge (v2);
}}
\tikzset{/graphs en masse/add graph={
  \node (v1) at (0.602,0.236) {};
  \node (v2) at (0.083,0.217) {};
  \node (v3) at (0.246,0.783) {};
  \node (v4) at (0.666,0.704) {};
  \node (v5) at (0.917,0.333) {} edge (v1);
  \node (v6) at (0.318,0.444) {} edge (v1) edge (v2) edge (v3);
}}
\tikzset{/graphs en masse/add graph={
  \node (v1) at (0.693,0.401) {};
  \node (v2) at (0.065,0.633) {};
  \node (v3) at (0.272,0.177) {};
  \node (v4) at (0.539,0.823) {};
  \node (v5) at (0.935,0.198) {} edge (v1);
  \node (v6) at (0.381,0.51) {} edge (v1) edge (v2) edge (v3) edge (v4);
}}
\tikzset{/graphs en masse/add graph={
  \node (v1) at (0.422,0.298) {};
  \node (v2) at (0.125,0.783) {};
  \node (v3) at (0.793,0.794) {};
  \node (v4) at (0.875,0.206) {};
  \node (v5) at (0.155,0.464) {} edge (v1) edge (v2);
  \node (v6) at (0.695,0.479) {} edge (v1) edge (v3) edge (v4);
}}
\tikzset{/graphs en masse/add graph={
  \node (v1) at (0.791,0.177) {};
  \node (v2) at (0.85,0.758) {};
  \node (v3) at (0.15,0.242) {};
  \node (v4) at (0.209,0.823) {};
  \node (v5) at (0.666,0.483) {} edge (v1) edge (v2);
  \node (v6) at (0.334,0.517) {} edge (v3) edge (v4) edge (v5);
}}
\tikzset{/graphs en masse/add graph={
  \node (v1) at (0.407,0.23) {};
  \node (v2) at (0.897,0.48) {};
  \node (v3) at (0.349,0.722) {};
  \node (v4) at (0.103,0.35) {} edge (v1);
  \node (v5) at (0.849,0.8) {} edge (v2);
  \node (v6) at (0.732,0.2) {} edge (v1) edge (v2);
}}
\tikzset{/graphs en masse/add graph={
  \node (v1) at (0.616,0.174) {};
  \node (v2) at (0.55,0.86) {};
  \node (v3) at (0.069,0.483) {};
  \node (v4) at (0.931,0.14) {} edge (v1);
  \node (v5) at (0.846,0.74) {} edge (v2);
  \node (v6) at (0.332,0.307) {} edge (v1) edge (v3);
}}
\tikzset{/graphs en masse/add graph={
  \node (v1) at (0.54,0.229) {};
  \node (v2) at (0.577,0.812) {};
  \node (v3) at (0.104,0.549) {};
  \node (v4) at (0.856,0.183) {} edge (v1);
  \node (v5) at (0.896,0.817) {} edge (v2);
  \node (v6) at (0.424,0.529) {} edge (v1) edge (v2) edge (v3);
}}
\tikzset{/graphs en masse/add graph={
  \node (v1) at (0.475,0.114) {};
  \node (v2) at (0.131,0.625) {};
  \node (v3) at (0.869,0.185) {};
  \node (v4) at (0.173,0.146) {} edge (v1);
  \node (v5) at (0.314,0.886) {} edge (v2);
  \node (v6) at (0.363,0.408) {} edge (v1) edge (v2) edge (v4);
}}
\tikzset{/graphs en masse/add graph={
  \node (v1) at (0.256,0.847) {};
  \node (v2) at (0.663,0.385) {};
  \node (v3) at (0.122,0.286) {};
  \node (v4) at (0.561,0.797) {} edge (v1);
  \node (v5) at (0.878,0.153) {} edge (v2);
  \node (v6) at (0.361,0.53) {} edge (v1) edge (v2) edge (v3) edge (v4);
}}
\tikzset{/graphs en masse/add graph={
  \node (v1) at (0.39,0.57) {};
  \node (v2) at (0.76,0.135) {};
  \node (v3) at (0.824,0.865) {};
  \node (v4) at (0.176,0.813) {} edge (v1);
  \node (v5) at (0.694,0.443) {} edge (v1) edge (v2);
  \node (v6) at (0.461,0.256) {} edge (v1) edge (v2) edge (v5);
}}
\tikzset{/graphs en masse/add graph={
  \node (v1) at (0.307,0.425) {};
  \node (v2) at (0.886,0.664) {};
  \node (v3) at (0.114,0.891) {};
  \node (v4) at (0.303,0.109) {} edge (v1);
  \node (v5) at (0.611,0.509) {} edge (v1) edge (v2);
  \node (v6) at (0.386,0.73) {} edge (v1) edge (v3) edge (v5);
}}
\tikzset{/graphs en masse/add graph={
  \node (v1) at (0.411,0.349) {};
  \node (v2) at (0.814,0.741) {};
  \node (v3) at (0.233,0.878) {};
  \node (v4) at (0.186,0.122) {} edge (v1);
  \node (v5) at (0.71,0.448) {} edge (v1) edge (v2);
  \node (v6) at (0.49,0.673) {} edge (v1) edge (v2) edge (v3) edge (v5);
}}
\tikzset{/graphs en masse/add graph={
  \node (v1) at (0.417,0.482) {};
  \node (v2) at (0.914,0.28) {};
  \node (v3) at (0.086,0.259) {};
  \node (v4) at (0.524,0.777) {} edge (v1);
  \node (v5) at (0.605,0.223) {} edge (v1) edge (v2);
  \node (v6) at (0.714,0.522) {} edge (v1) edge (v2) edge (v4) edge (v5);
}}
\tikzset{/graphs en masse/add graph={
  \node (v1) at (0.666,0.197) {};
  \node (v2) at (0.705,0.823) {};
  \node (v3) at (0.15,0.718) {};
  \node (v4) at (0.312,0.177) {} edge (v1);
  \node (v5) at (0.85,0.498) {} edge (v1) edge (v2);
  \node (v6) at (0.477,0.519) {} edge (v1) edge (v2) edge (v3) edge (v4) edge (v5);
}}
\tikzset{/graphs en masse/add graph={
  \node (v1) at (0.315,0.759) {};
  \node (v2) at (0.836,0.716) {};
  \node (v3) at (0.632,0.241) {};
  \node (v4) at (0.09,0.531) {} edge (v1);
  \node (v5) at (0.91,0.402) {} edge (v2) edge (v3);
  \node (v6) at (0.558,0.556) {} edge (v1) edge (v2) edge (v3);
}}
\tikzset{/graphs en masse/add graph={
  \node (v1) at (0.15,0.79) {};
  \node (v2) at (0.737,0.81) {};
  \node (v3) at (0.738,0.19) {};
  \node (v4) at (0.187,0.466) {} edge (v1);
  \node (v5) at (0.85,0.5) {} edge (v2) edge (v3);
  \node (v6) at (0.442,0.67) {} edge (v1) edge (v2) edge (v4);
}}
\tikzset{/graphs en masse/add graph={
  \node (v1) at (0.321,0.499) {};
  \node (v2) at (0.866,0.18) {};
  \node (v3) at (0.807,0.82) {};
  \node (v4) at (0.134,0.241) {} edge (v1);
  \node (v5) at (0.834,0.5) {} edge (v2) edge (v3);
  \node (v6) at (0.582,0.314) {} edge (v1) edge (v2) edge (v5);
}}
\tikzset{/graphs en masse/add graph={
  \node (v1) at (0.119,0.353) {};
  \node (v2) at (0.881,0.404) {};
  \node (v3) at (0.427,0.879) {};
  \node (v4) at (0.341,0.121) {} edge (v1);
  \node (v5) at (0.646,0.634) {} edge (v2) edge (v3);
  \node (v6) at (0.421,0.419) {} edge (v1) edge (v4) edge (v5);
}}
\tikzset{/graphs en masse/add graph={
  \node (v1) at (0.234,0.665) {};
  \node (v2) at (0.766,0.383) {};
  \node (v3) at (0.239,0.18) {};
  \node (v4) at (0.492,0.855) {} edge (v1);
  \node (v5) at (0.556,0.145) {} edge (v2) edge (v3);
  \node (v6) at (0.444,0.426) {} edge (v1) edge (v2) edge (v3) edge (v5);
}}
\tikzset{/graphs en masse/add graph={
  \node (v1) at (0.499,0.311) {};
  \node (v2) at (0.792,0.651) {};
  \node (v3) at (0.241,0.879) {};
  \node (v4) at (0.208,0.431) {} edge (v1);
  \node (v5) at (0.25,0.121) {} edge (v1) edge (v4);
  \node (v6) at (0.461,0.631) {} edge (v1) edge (v2) edge (v3) edge (v4);
}}
\tikzset{/graphs en masse/add graph={
  \node (v1) at (0.173,0.485) {};
  \node (v2) at (0.85,0.671) {};
  \node (v3) at (0.15,0.837) {};
  \node (v4) at (0.381,0.163) {} edge (v1);
  \node (v5) at (0.493,0.812) {} edge (v1) edge (v2) edge (v3);
  \node (v6) at (0.492,0.532) {} edge (v1) edge (v2) edge (v3) edge (v4);
}}
\tikzset{/graphs en masse/add graph={
  \node (v1) at (0.422,0.336) {};
  \node (v2) at (0.738,0.173) {};
  \node (v3) at (0.835,0.827) {};
  \node (v4) at (0.15,0.582) {} edge (v1);
  \node (v5) at (0.85,0.466) {} edge (v1) edge (v2) edge (v3);
  \node (v6) at (0.599,0.572) {} edge (v1) edge (v2) edge (v3) edge (v5);
}}
\tikzset{/graphs en masse/add graph={
  \node (v1) at (0.459,0.259) {};
  \node (v2) at (0.308,0.803) {};
  \node (v3) at (0.85,0.522) {};
  \node (v4) at (0.15,0.197) {} edge (v1);
  \node (v5) at (0.535,0.573) {} edge (v1) edge (v2) edge (v3);
  \node (v6) at (0.24,0.495) {} edge (v1) edge (v2) edge (v4) edge (v5);
}}
\tikzset{/graphs en masse/add graph={
  \node (v1) at (0.396,0.289) {};
  \node (v2) at (0.747,0.15) {};
  \node (v3) at (0.783,0.85) {};
  \node (v4) at (0.179,0.606) {} edge (v1);
  \node (v5) at (0.821,0.474) {} edge (v1) edge (v2) edge (v3);
  \node (v6) at (0.545,0.56) {} edge (v1) edge (v2) edge (v3) edge (v4) edge (v5);
}}
\tikzset{/graphs en masse/add graph={
  \node (v1) at (0.59,0.496) {};
  \node (v2) at (0.85,0.51) {};
  \node (v3) at (0.15,0.729) {};
  \node (v4) at (0.843,0.829) {} edge (v1) edge (v2);
  \node (v5) at (0.719,0.171) {} edge (v1) edge (v2);
  \node (v6) at (0.511,0.789) {} edge (v1) edge (v2) edge (v3);
}}
\tikzset{/graphs en masse/add graph={
  \node (v1) at (0.252,0.514) {};
  \node (v2) at (0.767,0.85) {};
  \node (v3) at (0.565,0.15) {};
  \node (v4) at (0.444,0.771) {} edge (v1) edge (v2);
  \node (v5) at (0.233,0.187) {} edge (v1) edge (v3);
  \node (v6) at (0.586,0.474) {} edge (v1) edge (v3) edge (v4);
}}
\tikzset{/graphs en masse/add graph={
  \node (v1) at (0.691,0.85) {};
  \node (v2) at (0.816,0.15) {};
  \node (v3) at (0.184,0.476) {};
  \node (v4) at (0.784,0.502) {} edge (v1) edge (v2);
  \node (v5) at (0.329,0.809) {} edge (v1) edge (v3);
  \node (v6) at (0.489,0.286) {} edge (v2) edge (v3) edge (v4);
}}
\tikzset{/graphs en masse/add graph={
  \node (v1) at (0.529,0.557) {};
  \node (v2) at (0.191,0.85) {};
  \node (v3) at (0.674,0.134) {};
  \node (v4) at (0.51,0.866) {} edge (v1) edge (v2);
  \node (v5) at (0.809,0.423) {} edge (v1) edge (v3);
  \node (v6) at (0.794,0.736) {} edge (v1) edge (v4) edge (v5);
}}
\tikzset{/graphs en masse/add graph={
  \node (v1) at (0.19,0.317) {};
  \node (v2) at (0.509,0.852) {};
  \node (v3) at (0.81,0.148) {};
  \node (v4) at (0.288,0.623) {} edge (v1) edge (v2);
  \node (v5) at (0.506,0.249) {} edge (v1) edge (v3);
  \node (v6) at (0.6,0.555) {} edge (v2) edge (v4) edge (v5);
}}
\tikzset{/graphs en masse/add graph={
  \node (v1) at (0.28,0.343) {};
  \node (v2) at (0.472,0.859) {};
  \node (v3) at (0.77,0.35) {};
  \node (v4) at (0.23,0.653) {} edge (v1) edge (v2);
  \node (v5) at (0.526,0.141) {} edge (v1) edge (v3);
  \node (v6) at (0.526,0.552) {} edge (v1) edge (v2) edge (v3) edge (v4);
}}
\tikzset{/graphs en masse/add graph={
  \node (v1) at (0.228,0.334) {};
  \node (v2) at (0.772,0.229) {};
  \node (v3) at (0.522,0.872) {};
  \node (v4) at (0.468,0.128) {} edge (v1) edge (v2);
  \node (v5) at (0.303,0.64) {} edge (v1) edge (v3);
  \node (v6) at (0.532,0.429) {} edge (v1) edge (v2) edge (v4) edge (v5);
}}
\tikzset{/graphs en masse/add graph={
  \node (v1) at (0.15,0.261) {};
  \node (v2) at (0.538,0.845) {};
  \node (v3) at (0.85,0.229) {};
  \node (v4) at (0.27,0.605) {} edge (v1) edge (v2);
  \node (v5) at (0.498,0.155) {} edge (v1) edge (v3);
  \node (v6) at (0.616,0.497) {} edge (v2) edge (v3) edge (v4) edge (v5);
}}
\tikzset{/graphs en masse/add graph={
  \node (v1) at (0.192,0.33) {};
  \node (v2) at (0.522,0.85) {};
  \node (v3) at (0.808,0.313) {};
  \node (v4) at (0.212,0.681) {} edge (v1) edge (v2);
  \node (v5) at (0.494,0.15) {} edge (v1) edge (v3);
  \node (v6) at (0.507,0.497) {} edge (v1) edge (v2) edge (v3) edge (v4) edge (v5);
}}
\tikzset{/graphs en masse/add graph={
  \node (v1) at (0.4,0.174) {};
  \node (v2) at (0.376,0.769) {};
  \node (v3) at (0.601,0.409) {} edge (v1);
  \node (v4) at (0.679,0.876) {} edge (v2);
  \node (v5) at (0.716,0.124) {} edge (v1) edge (v3);
  \node (v6) at (0.284,0.465) {} edge (v1) edge (v2) edge (v3);
}}
\tikzset{/graphs en masse/add graph={
  \node (v1) at (0.467,0.308) {};
  \node (v2) at (0.837,0.847) {};
  \node (v3) at (0.163,0.489) {} edge (v1);
  \node (v4) at (0.497,0.85) {} edge (v2);
  \node (v5) at (0.177,0.15) {} edge (v1) edge (v3);
  \node (v6) at (0.687,0.553) {} edge (v1) edge (v2) edge (v4);
}}
\tikzset{/graphs en masse/add graph={
  \node (v1) at (0.345,0.355) {};
  \node (v2) at (0.65,0.882) {};
  \node (v3) at (0.671,0.388) {} edge (v1);
  \node (v4) at (0.329,0.889) {} edge (v2);
  \node (v5) at (0.536,0.111) {} edge (v1) edge (v3);
  \node (v6) at (0.481,0.625) {} edge (v1) edge (v2) edge (v3) edge (v4);
}}
\tikzset{/graphs en masse/add graph={
  \node (v1) at (0.16,0.39) {};
  \node (v2) at (0.716,0.311) {};
  \node (v3) at (0.129,0.664) {} edge (v1);
  \node (v4) at (0.871,0.592) {} edge (v2);
  \node (v5) at (0.396,0.689) {} edge (v1) edge (v3);
  \node (v6) at (0.422,0.416) {} edge (v1) edge (v2) edge (v3) edge (v5);
}}
\tikzset{/graphs en masse/add graph={
  \node (v1) at (0.492,0.236) {};
  \node (v2) at (0.508,0.764) {};
  \node (v3) at (0.846,0.15) {} edge (v1);
  \node (v4) at (0.154,0.85) {} edge (v2);
  \node (v5) at (0.765,0.506) {} edge (v1) edge (v2) edge (v3);
  \node (v6) at (0.235,0.494) {} edge (v1) edge (v2) edge (v4);
}}
\tikzset{/graphs en masse/add graph={
  \node (v1) at (0.178,0.217) {};
  \node (v2) at (0.454,0.811) {};
  \node (v3) at (0.497,0.15) {} edge (v1);
  \node (v4) at (0.822,0.85) {} edge (v2);
  \node (v5) at (0.241,0.522) {} edge (v1) edge (v2) edge (v3);
  \node (v6) at (0.542,0.455) {} edge (v1) edge (v2) edge (v3) edge (v5);
}}
\tikzset{/graphs en masse/add graph={
  \node (v1) at (0.429,0.21) {};
  \node (v2) at (0.571,0.79) {};
  \node (v3) at (0.769,0.15) {} edge (v1);
  \node (v4) at (0.231,0.85) {} edge (v2);
  \node (v5) at (0.667,0.469) {} edge (v1) edge (v2) edge (v3);
  \node (v6) at (0.333,0.531) {} edge (v1) edge (v2) edge (v4) edge (v5);
}}
\tikzset{/graphs en masse/add graph={
  \node (v1) at (0.441,0.15) {};
  \node (v2) at (0.601,0.85) {};
  \node (v3) at (0.798,0.156) {} edge (v1);
  \node (v4) at (0.202,0.792) {} edge (v2);
  \node (v5) at (0.778,0.498) {} edge (v1) edge (v2) edge (v3);
  \node (v6) at (0.446,0.486) {} edge (v1) edge (v2) edge (v3) edge (v4) edge (v5);
}}
\tikzset{/graphs en masse/add graph={
  \node (v1) at (0.403,0.751) {};
  \node (v2) at (0.85,0.249) {};
  \node (v3) at (0.15,0.47) {} edge (v1);
  \node (v4) at (0.805,0.624) {} edge (v1) edge (v2);
  \node (v5) at (0.546,0.505) {} edge (v1) edge (v2) edge (v3) edge (v4);
  \node (v6) at (0.474,0.277) {} edge (v1) edge (v2) edge (v3) edge (v4) edge (v5);
}}
\tikzset{/graphs en masse/add graph={
  \node (v1) at (0.497,0.271) {};
  \node (v2) at (0.15,0.842) {};
  \node (v3) at (0.729,0.509) {} edge (v1);
  \node (v4) at (0.85,0.158) {} edge (v1) edge (v3);
  \node (v5) at (0.271,0.491) {} edge (v1) edge (v2) edge (v3);
  \node (v6) at (0.503,0.729) {} edge (v1) edge (v2) edge (v3) edge (v5);
}}
\tikzset{/graphs en masse/add graph={
  \node (v1) at (0.187,0.514) {};
  \node (v2) at (0.813,0.15) {};
  \node (v3) at (0.252,0.823) {} edge (v1);
  \node (v4) at (0.589,0.85) {} edge (v1) edge (v3);
  \node (v5) at (0.678,0.556) {} edge (v1) edge (v2) edge (v3);
  \node (v6) at (0.44,0.34) {} edge (v1) edge (v2) edge (v3) edge (v4) edge (v5);
}}
\tikzset{/graphs en masse/add graph={
  \node (v1) at (0.187,0.373) {};
  \node (v2) at (0.85,0.199) {};
  \node (v3) at (0.15,0.658) {} edge (v1);
  \node (v4) at (0.401,0.801) {} edge (v1) edge (v3);
  \node (v5) at (0.6,0.596) {} edge (v1) edge (v3) edge (v4);
  \node (v6) at (0.475,0.334) {} edge (v1) edge (v2) edge (v3) edge (v4);
}}
\tikzset{/graphs en masse/add graph={
  \node (v1) at (0.238,0.287) {};
  \node (v2) at (0.762,0.713) {};
  \node (v3) at (0.607,0.367) {} edge (v1) edge (v2);
  \node (v4) at (0.189,0.889) {} edge (v1) edge (v2);
  \node (v5) at (0.811,0.111) {} edge (v1) edge (v2) edge (v3);
  \node (v6) at (0.393,0.633) {} edge (v1) edge (v2) edge (v3) edge (v4);
}}
\tikzset{/graphs en masse/add graph={
  \node (v1) at (0.225,0.368) {};
  \node (v2) at (0.666,0.276) {};
  \node (v3) at (0.557,0.836) {};
  \node (v4) at (0.95,0.612) {};
  \node (v5) at (0.05,0.582) {} edge (v1);
  \node (v6) at (0.414,0.164) {} edge (v1) edge (v2);
  \node (v7) at (0.671,0.569) {} edge (v2) edge (v3) edge (v4);
}}
\tikzset{/graphs en masse/add graph={
  \node (v1) at (0.848,0.513) {};
  \node (v2) at (0.569,0.051) {};
  \node (v3) at (0.148,0.504) {};
  \node (v4) at (0.231,0.949) {};
  \node (v5) at (0.744,0.812) {} edge (v1);
  \node (v6) at (0.32,0.243) {} edge (v2) edge (v3);
  \node (v7) at (0.852,0.196) {} edge (v1) edge (v2);
}}
\tikzset{/graphs en masse/add graph={
  \node (v1) at (0.402,0.27) {};
  \node (v2) at (0.906,0.617) {};
  \node (v3) at (0.353,0.911) {};
  \node (v4) at (0.882,0.089) {};
  \node (v5) at (0.094,0.342) {} edge (v1);
  \node (v6) at (0.658,0.817) {} edge (v2) edge (v3);
  \node (v7) at (0.71,0.36) {} edge (v1) edge (v2) edge (v4);
}}
\tikzset{/graphs en masse/add graph={
  \node (v1) at (0.232,0.557) {};
  \node (v2) at (0.783,0.57) {};
  \node (v3) at (0.518,0.114) {};
  \node (v4) at (0.193,0.872) {} edge (v1);
  \node (v5) at (0.807,0.886) {} edge (v2);
  \node (v6) at (0.226,0.24) {} edge (v1) edge (v3);
  \node (v7) at (0.803,0.254) {} edge (v2) edge (v3);
}}
\tikzset{/graphs en masse/add graph={
  \node (v1) at (0.087,0.531) {};
  \node (v2) at (0.659,0.771) {};
  \node (v3) at (0.535,0.149) {};
  \node (v4) at (0.108,0.851) {} edge (v1);
  \node (v5) at (0.913,0.587) {} edge (v2);
  \node (v6) at (0.231,0.246) {} edge (v1) edge (v3);
  \node (v7) at (0.634,0.452) {} edge (v2) edge (v3) edge (v5);
}}
\tikzset{/graphs en masse/add graph={
  \node (v1) at (0.793,0.364) {};
  \node (v2) at (0.246,0.272) {};
  \node (v3) at (0.336,0.853) {};
  \node (v4) at (0.848,0.685) {} edge (v1);
  \node (v5) at (0.5,0.445) {} edge (v1) edge (v2);
  \node (v6) at (0.152,0.586) {} edge (v2) edge (v3);
  \node (v7) at (0.547,0.147) {} edge (v1) edge (v2) edge (v5);
}}
\tikzset{/graphs en masse/add graph={
  \node (v1) at (0.655,0.822) {};
  \node (v2) at (0.721,0.276) {};
  \node (v3) at (0.15,0.373) {};
  \node (v4) at (0.338,0.807) {} edge (v1);
  \node (v5) at (0.85,0.566) {} edge (v1) edge (v2);
  \node (v6) at (0.411,0.178) {} edge (v2) edge (v3);
  \node (v7) at (0.536,0.543) {} edge (v1) edge (v2) edge (v4) edge (v5);
}}
\tikzset{/graphs en masse/add graph={
  \node (v1) at (0.576,0.82) {};
  \node (v2) at (0.722,0.384) {};
  \node (v3) at (0.139,0.18) {};
  \node (v4) at (0.259,0.805) {} edge (v1);
  \node (v5) at (0.861,0.67) {} edge (v1) edge (v2);
  \node (v6) at (0.454,0.226) {} edge (v2) edge (v3);
  \node (v7) at (0.437,0.54) {} edge (v1) edge (v2) edge (v4) edge (v6);
}}
\tikzset{/graphs en masse/add graph={
  \node (v1) at (0.571,0.709) {};
  \node (v2) at (0.786,0.231) {};
  \node (v3) at (0.15,0.23) {};
  \node (v4) at (0.268,0.803) {} edge (v1);
  \node (v5) at (0.85,0.548) {} edge (v1) edge (v2);
  \node (v6) at (0.466,0.197) {} edge (v2) edge (v3);
  \node (v7) at (0.333,0.491) {} edge (v1) edge (v3) edge (v4) edge (v6);
}}
\tikzset{/graphs en masse/add graph={
  \node (v1) at (0.525,0.801) {};
  \node (v2) at (0.795,0.331) {};
  \node (v3) at (0.2,0.134) {};
  \node (v4) at (0.214,0.866) {} edge (v1);
  \node (v5) at (0.8,0.645) {} edge (v1) edge (v2);
  \node (v6) at (0.513,0.185) {} edge (v2) edge (v3);
  \node (v7) at (0.528,0.493) {} edge (v1) edge (v2) edge (v5) edge (v6);
}}
\tikzset{/graphs en masse/add graph={
  \node (v1) at (0.693,0.304) {};
  \node (v2) at (0.21,0.37) {};
  \node (v3) at (0.461,0.868) {};
  \node (v4) at (0.798,0.6) {} edge (v1);
  \node (v5) at (0.424,0.132) {} edge (v1) edge (v2);
  \node (v6) at (0.202,0.687) {} edge (v2) edge (v3);
  \node (v7) at (0.481,0.546) {} edge (v1) edge (v2) edge (v3) edge (v4) edge (v6);
}}
\tikzset{/graphs en masse/add graph={
  \node (v1) at (0.773,0.493) {};
  \node (v2) at (0.322,0.201) {};
  \node (v3) at (0.273,0.752) {};
  \node (v4) at (0.851,0.799) {} edge (v1);
  \node (v5) at (0.633,0.215) {} edge (v1) edge (v2);
  \node (v6) at (0.149,0.464) {} edge (v2) edge (v3);
  \node (v7) at (0.459,0.486) {} edge (v1) edge (v2) edge (v3) edge (v5) edge (v6);
}}
\tikzset{/graphs en masse/add graph={
  \node (v1) at (0.358,0.172) {};
  \node (v2) at (0.85,0.46) {};
  \node (v3) at (0.401,0.828) {};
  \node (v4) at (0.15,0.42) {} edge (v1);
  \node (v5) at (0.681,0.185) {} edge (v1) edge (v2);
  \node (v6) at (0.717,0.754) {} edge (v2) edge (v3);
  \node (v7) at (0.5,0.486) {} edge (v1) edge (v2) edge (v3) edge (v4) edge (v5) edge (v6);
}}
\tikzset{/graphs en masse/add graph={
  \node (v1) at (0.774,0.416) {};
  \node (v2) at (0.222,0.44) {};
  \node (v3) at (0.542,0.894) {};
  \node (v4) at (0.778,0.106) {} edge (v1);
  \node (v5) at (0.254,0.759) {} edge (v2) edge (v3);
  \node (v6) at (0.496,0.264) {} edge (v1) edge (v2) edge (v4);
  \node (v7) at (0.504,0.577) {} edge (v1) edge (v2) edge (v3) edge (v5) edge (v6);
}}
\tikzset{/graphs en masse/add graph={
  \node (v1) at (0.417,0.75) {};
  \node (v2) at (0.27,0.211) {};
  \node (v3) at (0.879,0.299) {};
  \node (v4) at (0.121,0.846) {} edge (v1);
  \node (v5) at (0.59,0.154) {} edge (v2) edge (v3);
  \node (v6) at (0.473,0.443) {} edge (v1) edge (v2) edge (v5);
  \node (v7) at (0.183,0.531) {} edge (v1) edge (v2) edge (v4) edge (v6);
}}
\tikzset{/graphs en masse/add graph={
  \node (v1) at (0.564,0.233) {};
  \node (v2) at (0.831,0.81) {};
  \node (v3) at (0.169,0.769) {};
  \node (v4) at (0.243,0.15) {} edge (v1);
  \node (v5) at (0.497,0.85) {} edge (v2) edge (v3);
  \node (v6) at (0.642,0.547) {} edge (v1) edge (v2) edge (v5);
  \node (v7) at (0.32,0.467) {} edge (v1) edge (v3) edge (v4) edge (v6);
}}
\tikzset{/graphs en masse/add graph={
  \node (v1) at (0.392,0.403) {};
  \node (v2) at (0.841,0.597) {};
  \node (v3) at (0.283,0.894) {};
  \node (v4) at (0.465,0.106) {} edge (v1);
  \node (v5) at (0.159,0.2) {} edge (v1) edge (v4);
  \node (v6) at (0.592,0.804) {} edge (v2) edge (v3);
  \node (v7) at (0.701,0.313) {} edge (v1) edge (v2) edge (v4);
}}
\tikzset{/graphs en masse/add graph={
  \node (v1) at (0.359,0.38) {};
  \node (v2) at (0.712,0.743) {};
  \node (v3) at (0.125,0.579) {} edge (v1);
  \node (v4) at (0.903,0.482) {} edge (v2);
  \node (v5) at (0.097,0.257) {} edge (v1) edge (v3);
  \node (v6) at (0.615,0.451) {} edge (v1) edge (v2) edge (v4);
  \node (v7) at (0.408,0.697) {} edge (v1) edge (v2) edge (v3) edge (v6);
}}
\tikzset{/graphs en masse/add graph={
  \node (v1) at (0.5,0.764) {};
  \node (v2) at (0.581,0.079) {};
  \node (v3) at (0.95,0.521) {};
  \node (v4) at (0.05,0.321) {};
  \node (v5) at (0.25,0.921) {} edge (v1);
  \node (v6) at (0.807,0.265) {} edge (v2) edge (v3);
  \node (v7) at (0.335,0.242) {} edge (v2) edge (v4);
  \node (v8) at (0.795,0.772) {} edge (v1) edge (v3);
}}